\documentclass[a4paper,11pt]{amsart}
 \usepackage[arrow,matrix]{xy}
\usepackage{amsmath,amssymb,amscd,bbm,amsthm,mathrsfs,dsfont}
\usepackage[usenames,dvipsnames]{pstricks}
\newtheorem{thm}{Theorem}[section]
\newtheorem{lem}{Lemma}[section]
\newtheorem{cor}{Corollary}[section]
\newtheorem{prop}{Proposition}[section]

\theoremstyle{definition}

\begin{document}
\numberwithin{equation}{section}

\title[ Hodge-type    decomposition and
cohomolgy groups of $k$-Cauchy-Fueter complexes ]
{On the Hodge-type    decomposition and
cohomolgy groups of $k$-Cauchy-Fueter complexes over domains in the
quaternionic space }
\author {Der-Chen Chang, Irina Markina  and Wei Wang
 }
\begin{abstract} The $k$-Cauchy-Fueter operator  $ D_0^{(k) } $ on one dimensional quaternionic space $\mathbb{H}$  is the Euclidean version  of  helicity
$\frac k 2$  massless field operator  on the Minkowski space in physics.  The $k$-Cauchy-Fueter equation  for $k\geq 2$ is  overdetermined and its compatibility
condition   is given by the $k$-Cauchy-Fueter complex. In quaternionic analysis, these complexes  play  the role  of Dolbeault complex   in  several complex variables.  We prove that
a natural boundary value problem associated to this complex  is regular. Then by using the theory of regular boundary value problems, we show the
Hodge-type orthogonal  decomposition, and the fact
that the  non-homogeneous $k$-Cauchy-Fueter equation  $ D_0^{(k) } u=f$ on a smooth domain $\Omega$ in $\mathbb{H}$ is  solvable if and only if $f$
satisfies the compatibility  condition and
is orthogonal to the set $\mathscr H^1_{ (k) }(\Omega)$ of Hodge-type elements. This set is isomorphic to the first   cohomology group  of the
$k$-Cauchy-Fueter complex over $\Omega$, which  is finite dimensional, while the second   cohomology group is always trivial.
\end{abstract}

\address{Department of Mathematics and Department of Computer Science,
Georgetown University, Washington D.C. 20057, USA
\newline
Department of Mathematics, Fu Jen Catholic University, Taipei 242, Taiwan, ROC}
\email{chang@georgetown.edu}

\address{Department of Mathematics, University of Bergen, NO-5008 Bergen, Norway}
\email{irina.markina@uib.no}

\address{Department of Mathematics, Zhejiang University, Zhejiang 310028, PR China}
\email{wwang@zju.edu.cn.}

\thanks{The first author is partially supported by an NSF grant DMS-1203845 and Hong Kong
RGC competitive earmarked research grant $\#$601410. The second and the third authors gratefully acknowledge partial support by the grants NFR-204726/V30
and NFR-213440/BG, Norwegian Research Council. The third author is also partially supported by National Nature Science Foundation in China (No. 11171298).
}
\maketitle
\section{\bf Introduction}
               On one dimensional quaternionic space, the $k$-Cauchy-Fueter operator
is the Euclidean version of  helicity $\frac k 2$  massless field operator  \cite{Eastwood}  \cite{We} on the Minkowski space in physics
 (corresponding to   the Dirac-Weyl equation for $k=1$,   Maxwell's equation for $k=2$,    the linearized Einstein's equation for $k=3$, etc.).
 They are the   quaternionic counterpart of the
Cauchy-Riemann operator in complex analysis.  In the quaternionic case, we have a family of operators  acting on  $ \odot^{k} \mathbb{C}^2 $-valued functions, because   we have a family of irreducible representations $ \odot^{k} \mathbb{C}^2 $ of
 SU$(2)$ ($=$ the group of unit quaternions), while $\mathbb{C}$ has only one irreducible representation.

The $k$-Cauchy-Fueter equation   is usually  overdetermined and its compatibility
condition   is given by the $k$-Cauchy-Fueter complex.
The $k$-Cauchy-Fueter complex  on multidimensional quaternionic space ${\mathbb H}^n$, which
  plays the role  of Dolbeault complex   in  several complex variables,  is now explicitly known \cite{Wa10} (cf. also \cite{Ba} for the existence and  \cite{bures} \cite{CSS}
  \cite{CSSS} for $k=1$).
It is quite interesting to  develop a theory of several
quaternionic variables by analyzing these complexes, as it was done for the Dolbeault complex in the theory of several complex variables. A well known
theorem in several complex variables states that  the Dolbeault cohomology of a  domain vanishes if and only if it is pseudoconvex. Many remarkable
results about holomorphic functions can be deduced by considering non-homogeneous $\overline{\partial}$-equations, which leads to
the study of $\overline{\partial}$-Neumann problem (cf., e.g., \cite{CNS} \cite{CS}).      We have solved \cite{Wa10} the  non-homogeneous
$k$-Cauchy-Fueter equation  on  the whole
quaternionic space  ${\mathbb H}^n$ and deduced Hartogs' phenomenon and integral representation formulae. See \cite{KW} \cite{LZ}  \cite{WR} \cite{Wa10}
(also \cite{adams2} \cite{CM08}  \cite{CMW} \cite{CSSS} \cite{Wa11} for $k=1$) and references therein for results about
  $k$-regular functions.

Note that the non-homogeneous $\overline{\partial}$-equation on a smooth domain in the complex plane is always solvable. In our case the non-homogeneous
$1$-Cauchy-Fueter equation   on a smooth domain in
  $\mathbb{H}$ is    always solvable since it is exactly the Dirac equation on $ \mathbb{ R}^4$. But even on one dimensional quaternionic space
  $\mathbb{H}$, the  $k$-Cauchy-Fueter operator for $k\geq 2$ is overdetermined. The  non-homogeneous
$k$-Cauchy-Fueter equation  only can be solved under the   compatibility condition given by the $k$-Cauchy-Fueter complex.
The $k$-Cauchy-Fueter complex over  a smooth domain $\Omega$ in $\mathbb{H}$ is
\begin{equation}\label{eq:k-CF-1}\begin{split}
 0\longrightarrow C^\infty (\Omega,  \mathbb{C}^{k+1} )
 \xrightarrow{D_0^{(k) } }  C^\infty (\Omega,
  \mathbb{C}^{2k })\xrightarrow{D_1^{(k) } }C^\infty (\Omega,  \mathbb{C}^{k -1} )\longrightarrow  0,\end{split}
\end{equation} $  k=2,3,\ldots$,
where $D_0^{(k) } $ is the  $k$-Cauchy-Fueter operator. In this paper, we will investigate
the  non-homogeneous $k$-Cauchy-Fueter equation
\begin{equation}\label{eq:k-CF-1-eq}
   D_0^{(k) } u=f,
\end{equation}on a smooth domain $\Omega$ in   $\mathbb{H}$ under the compatibility  condition
\begin{equation}\label{eq:compatibility}
     D_1^{(k) }f=0.
\end{equation}

We define the first cohomology group of the $k$-Cauchy-Fueter complex as
\begin{equation*}
    H^1_{ (k) }(\Omega)=\frac {\left\{f\in C^\infty (\overline{\Omega};\mathbb{C}^{2k });D_1^{(k) }f=0\right\}}
    {\left\{D_0^{(k)}u; u \in C^\infty (\overline{\Omega};\mathbb{C}^{ k+1 }) \right\}},
\end{equation*}
where $\overline{\Omega}$ is the closure of $ {\Omega}$, and the second cohomology group as
\begin{equation*}
    H^2_{ (k) }(\Omega)=\frac {   C^\infty (\overline{\Omega}; \mathbb{C}^{k -1} ) }
    {\left\{D_1^{(k)}u; u \in C^\infty (\overline{\Omega};\mathbb{C}^{2k }) \right\}}
.
\end{equation*}
The $0$-th cohomology group as
$
    H^0_{ (k) }(\Omega)=\ker D_0^{(k)}$. This is the space of $k$-regular functions, the dimension of which is infinite (cf. \cite{KW}).

    The first  cohomology group  can be represented by Hodge-type elements:
   \begin{equation*}
   \mathscr H^1_{ (k) }(\Omega)= \left\{f \in C^\infty (\Omega,
  \mathbb{C}^{2k }); D_1^{(k) }f=0, D_0^{(k)* }f=0 \right\}
,
\end{equation*}
where $ D_0^{(k)* }$ is the formal adjoint of $ D_0^{(k) }$.

Let $H^s(\Omega)$ be the Sobolev space of complex valued functions, defined on a domain $\Omega$. Denote by $H^s(\Omega, \mathbb{C}^n)$   the   space of
all $\mathbb{C}^n$-valued functions, whose components are in $H^s(\Omega)$.

\begin{thm} \label{thm:k-CF} Suppose $\Omega$ is a domain in $\mathbb{H} $ with smooth boundary. Then
\item[(1)] {the isomorphic spaces\begin{equation*}
H^1_{ (k) }(\Omega)\cong  \mathscr H^1_{ (k) }(\Omega)
\end{equation*}
are  finite dimensional;}

\item[(2)] {if $f\in H^s(\Omega,  \mathbb{C}^{2k } )$ ($s= 1,2,\ldots$), then
the  non-homogeneous $k$-Cauchy-Fueter equation (\ref{eq:k-CF-1-eq}) is solvable by some $u\in H^{s+1}(\Omega,  \mathbb{C}^{ k +1} )$ if and only if $f$
is orthogonal to  $\mathscr H^1_{ (k) }(\Omega)$ in $L^2(\Omega,  \mathbb{C}^{ 2k  } )$   and satisfies the compatibility  condition
(\ref{eq:compatibility}).
When it is solvable, it has a solution $u $ satisfying the estimate
\begin{equation}\label{eq:estimate1.1}
     \|u\|_{H^{s+1}(\Omega,  \mathbb{C}^{ k +1} )}\leq C\|f\|_{H^s(\Omega,  \mathbb{C}^{2k } )},
\end{equation}
for some constant $C$ only depending on the domain $\Omega$, $k$ and $s$;}

\item[(3)]{the equation
\begin{equation*}
     D_1^{(k)
      }\psi=\Psi,
\end{equation*} is uniquely  solved by a $\psi\in H^{s+1}(\Omega,  \mathbb{C}^{ 2k } )$ for any $\Psi\in H^s(\Omega,  \mathbb{C}^{ k -1} )$ with estimate
as (\ref{eq:estimate1.1}).}
\end{thm}

It follows from  Theorem 1.1 (3) and  elliptic regularity that the second  cohomology group always vanishes.
To prove Theorem \ref{thm:k-CF}, we consider the associated Laplacian of the complex (\ref{eq:k-CF-1})
\begin{equation}\label{eq:laplacian}
   \square_1^{(k)}£º=D_0^{(k) }D_0^{(k) *}+  D_1^{(k) *} D_1^{(k) },
\end{equation}
   where                 $D_0^{(k) *}$ and $D_1^{(k) *}$ be the formal adjoints of $D_0^{(k) } $ and $D_1^{(k) } $,
   respectively,
and a natural boundary value problem
\begin{equation}\label{eq:bvp}
  \left\{\begin{array}{l} \square_1^{(k)}u=f,\qquad \qquad\qquad\qquad{\rm on } \qquad\Omega,\\
\left.D_0^{(k) *}(\nu) u\right|_{\partial \Omega}=0,\\\left. D_1^{(k) *}(\nu)D_1^{(k) }u\right|_{\partial \Omega}=0,
 \end{array} \right.
\end{equation}
  where $\nu $ is the unit vector of outer normal to the boundary $\partial\Omega$, $u\in H^{s+2}(\Omega,  \mathbb{C}^{ 2k  } )$ and $f\in H^{s }(\Omega,
  \mathbb{C}^{2k } )$. We   prove that this boundary value problem is regular and obtain the following result.

\begin{thm} \label{thm:BVP}Suppose $\Omega$ is a domain in $\mathbb{H} $ with a smooth boundary.  If $f\in H^s(\Omega,  \mathbb{C}^{2k } )$ ($s=0,
1,2,\ldots$) is orthogonal to  $ \mathscr  H^1_{ (k) }(\Omega)$ relative to the  $L^2$ inner product, the boundary value problem (\ref{eq:bvp})
has a   solution $u=N_1^{(k) }f$ such that
\begin{equation}\label{eq:estimate}
     \|u\|_{H^{s+2}(\Omega,  \mathbb{C}^{2k } )}\leq C\|f\|_{H^s(\Omega,  \mathbb{C}^{2k } )}
\end{equation}
for some constant $C$ only depending on the domain $\Omega$, $k$ and $s$.

Moreover, we have the Hodge-type orthogonal  decomposition for any $\psi\in H^s(\Omega,  \mathbb{C}^{2k } )$:
\begin{equation}\label{eq:Hodge-decomposition}
     \psi= D_0^{(k) }D_0^{(k)* } N_1^{(k) }\psi+D_1^{(k) *}D_1^{(k) }N_1^{(k) }\psi +P \psi,
\end{equation} where $P $ is the orthonomal projection to $ \mathscr  H^1_{ (k) }(\Omega)$ under the  $L^2(\Omega,  \mathbb{C}^{2k } )$ inner product.
\end{thm}
Although for a smooth domain in the complex plane, its
Dolbeault cohomology       always vanishes,
its De Rham cohomolgy groups, which are isomorphic to its simplicial cohomolog groups,  may be nontrivial. We conjecture that the cohomology groups $H^1_{
(k) }(\Omega)$   may  be nontrivial for some domains $\Omega$ with  smooth boundaries  in $\mathbb{H}$. It is quite interesting to characterize the class
of domains in $\mathbb{H}$ on which the non-homogeneous $k$-Cauchy-Fueter equation is always solvable.    On the higher dimensional quaternionic space
$\mathbb{H}^n$, there is no reason to expect the corresponding boundary value problem  of the non-homogeneous $k$-Cauchy-Fueter equation  to be regular,
as in the case of several complex variables. The problem becomes much harder. It is also interesting to find some $L^2$ estimates for the
$k$-Cauchy-Fueter equation on  a   domain in $\mathbb{H}^n$.

In section 2, we will write the $2 $-Cauchy-Fueter operator $D_0^{(2)}$ and  the operator $D_1^{(2)}$ explicitly as a $(4\times 3)$-matrix and a $(1\times
4)$-matrix valued differential operators of first order with constant coefficients, respectively, and calculate  the associated Laplacian. We also find
the natural boundary conditions for functions in   domains of the adjoint operator $D_0^{(2)*}$ or    $D_1^{(2)*}$.  In section 3, we prove that the
boundary value problem
(\ref{eq:bvp}) satisfies the  Shapiro-Lopatinskii condition, i.e., it is a regular boundary value problem.
In section 4, we generalize the results of sections 2 and 3 to the cases $k\geq 3$. The   $k$-Cauchy-Fueter operator $D_0^{(k)}$ and  the second operator
$D_1^{(k)}$ in the complex (\ref{eq:k-CF-1}) are written explicitly as  matrix valued differential operators of first order with constant coefficients,
the associated Laplacians are calculated, and  the boundary value problem
  is  proved to be also  regular. In section 5, we apply the general theory for elliptic boundary value problems to show that $ \square_1^{(k)}$ is a
  Fredholm operator between suitable Sobolev spaces. This implies the Hodge-type decomposition and allows us to prove   main theorems.

Because we only work on one dimensional quaternionic space, the resulst in \cite{Wa10} about the $k$-Cauchy-Fueter complex, that we will use later, can
be proved by elementary
method. So this paper is self-contained.

\section{\bf The $k$-Cauchy-Fueter operators}

\subsection{The $k$-Cauchy-Fueter complexes on a domain in   ${\mathbb H} $}
We will identify the one dimensional quaternionic space $\mathbb{H}$ with the Euclidean space $\mathbb{R}^4$, setting
 \begin{equation} \label{eq:Cauchy-Fueter1}
   \left( \begin{array}{cc} \nabla_{00'}&
  \nabla_{01'}\\\nabla_{10'}&  \nabla_{11'}\end{array}\right):=
  \left( \begin{array}{cc} \partial_{x_{0}}+ {i} \partial_{x_{ 1}}&
 -\partial_{x_{  2}}- {i}\partial_{x_{ 3}}\\\partial_{x_{ 2}}-
 {i}\partial_{ x_{ 3}}&
 \partial_{ x_{0 }}- {i}\partial_{  x_{ 1}}\end{array}\right),
 \end{equation} where $(x_{0},x_{1},x_{2},x_{3})\in \mathbb{R}^4$. The
 matrix
 \begin{equation}\epsilon= (\epsilon_{A'B'})=\left( \begin{array}{cc} 0&
 1\\-1& 0\end{array}\right)\label{eq:epsilon}
 \end{equation}
 is used to raise or lower indices, e.g. $
  \nabla_{A}^{A'}\epsilon_{A'B'}=\nabla_{AB'}$.

 The {\it $k$-Cauchy-Fueter complex} \cite{Wa10} on  a domain $\Omega$ in $\mathbb{R}^4$ for $ k\geq2$ is
 \begin{equation}\label{eq:Dirac-Wey-k}\begin{split}
  0\longrightarrow C^\infty (\Omega, \odot^{k }\mathbb{C}^2 )
  \xrightarrow{D_0^{(k) } }  C^\infty (\Omega,
 \odot^{k -1} \mathbb{C}^2\otimes \mathbb{C}^{2 })\xrightarrow{D_1^{(k) } }C^\infty (\Omega,\odot^{k -2}
 \mathbb{C}^2\otimes  \Lambda^2\mathbb{C}^2 )\longrightarrow  0,\end{split}
 \end{equation}
where
$\odot^{k} \Bbb C^{2}$ is  the $k$-th symmetric power of $  \Bbb C^{2}$,
 \begin{equation}\label{eq:Dirac-Wey-k-0}\begin{split}
     & (D_0^{(k) } \phi)_{A B'\cdots C' }:=\sum_{A'=0',1'}\nabla^{A'}_A\phi_{A'B'\cdots C'  },\\& (D_1^{(k) } \psi)_{A B  B'\cdots
     C'}:=\sum_{A'=0',1'}\left(\nabla^{A'}_{ A}\psi_{
     B A'B'\cdots C'  }-\nabla^{A'}_{B}\psi_{AA'B'\cdots C'  }\right ) .
 \end{split}\end{equation}
Here a section $\phi\in
 C^\infty(\Omega,\odot^{k }\mathbb{C}^2 )$ has $(k+1)$ components
 $\phi_{0'\ldots 0'},\phi_{1'\ldots 0'}  ,  \ldots ,\phi_{1'\ldots
 1'}$, while $D_0^{(k)}\phi\in C^\infty(\Omega,\odot^{k-1
 }\mathbb{C}^2\otimes \mathbb{C}^{2 })$ has $2 k$  components
 $(D_0^{(k)}\phi)_{A0'\ldots 0'}$, $(D_0^{(k)}\phi)_{A1'\ldots 0'} ,
 \ldots ,(D_0^{(k)}\phi)_{A1'\ldots  1'}$, where $A=0,1$. Note that  $\phi_{A'B'\ldots C'}$ is  invariant
under
 the permutation of subscripts, $A',B',\cdots, C'=0',1'$.

There are a family of equations in physics, called the
{\it helicity $\frac k 2$  massless field} equations~\cite{Eastwood} \cite{We}.
   The first one is the Dirac-Weyl  equation of
an electron  for mass zero whose solutions correspond to neutrinos.
 The second one is the Maxwell's equation  whose  solutions correspond to
photons.  The third one is the linearized Einstein's equation
 whose  solutions correspond to  weak gravitational field, and so on. The $k$-Cauchy-Fueter equations
 are the  Euclidean version of these equations.  The {\it affine
Minkowski space} can be  embedded in $\mathbb{C}^{2\times 2 }  $ by
\begin{equation}\label{eq:embed-M}( x_0,  x_1, x_2,x_3)\mapsto
\left( \begin{array} {cc} x_0+  x_1&  x_2+ix_3\\x_2-ix_3&  x_0-
x_1\end{array}\right),
\end{equation} ${i}=\sqrt{-1},$
while the quaternionic  algebra $\mathbb H $  can be  embedded in
$\mathbb{C}^{2\times 2 }  $ by
\begin{equation}\label{eq:embed-Q1}x_0+  x_1\mathbf{i}+ x_2\mathbf{j}
+x_3\mathbf{k}\mapsto\left( \begin{array} {cc} x_0+ {i} x_1&
-x_2-ix_3\\x_2-ix_3&  x_0- {i} x_1\end{array}\right).
\end{equation}
The
  helicity $\frac k 2$  massless field  equation (cf.~\cite{Eastwood}~\cite{Wa10}) is
  \begin{equation*}
          D_0^{(k) } \phi =0,\end{equation*}
where the $D_0^{(k) } $ is also given by (\ref{eq:Dirac-Wey-k-0}) with $\nabla_{AB'}$ replaced  by
\begin{equation} \label{eq:helicity}
   \left( \begin{array}{cc} \nabla_{00'}&                                                   \nabla_{01'}\\\nabla_{10'}&  \nabla_{11'}\end{array}\right):=
  \left( \begin{array}{cc} \partial_{x_{0}}+   \partial_{x_{ 1}}&
  \partial_{x_{  2}}+ {i}\partial_{x_{ 3}}\\\partial_{x_{ 2}}-
 {i}\partial_{ x_{ 3}}&
 \partial_{ x_{0 }}-  \partial_{  x_{ 1}}\end{array}\right).
 \end{equation}

\subsection{The $2$-Cauchy-Fueter complex }
We write
\begin{equation}\begin{split}  \left(\begin{array}{cc } \nabla_0^{0'}&
\nabla_0^{1'} \\\nabla_1^{0'}& \nabla_1^{1'}\end{array} \right)
&=\left( \begin{array}{cc} \nabla_{00'}&
 \nabla_{01'}\\\nabla_{10'}&  \nabla_{11'}\end{array}\right)\left( \begin{array}{cc} 0&
-1\\1& 0\end{array}\right) = \left(\begin{array}{cc } \nabla_{01'}&-\nabla_{00'}
\\ \nabla_{11'}&  -\nabla_{10'}\end{array} \right)  \\&
=\left( \begin{array}{cc}
-\partial_{x_{  2}}- {i}\partial_{x_{ 3}}&-\partial_{x_{0}}- {i} \partial_{x_{ 1}}\\
\partial_{ x_{0 }}- {i}\partial_{  x_{ 1}}&-\partial_{x_{ 2}}+
{i}\partial_{ x_{ 3}}\end{array}\right).\label{eq:relations-index}
\end{split}\end{equation}

In the case $  k=2$, we use the notation $D_0=D_0^{(2)}$ and $D_1=D_1^{(2)}$.
The $2$-Cauchy-Fueter complex on  a domain $\Omega$ in $\mathbb{R}^4$ is
\begin{equation}\label{eq:Dirac-Wey}\begin{split}
 0\longrightarrow C^\infty (\Omega, \odot^{2 }\mathbb{C}^2 )
 \xrightarrow{D_0 }  C^\infty (\Omega,
 \mathbb{C}^2\otimes \mathbb{C}^{2 })\xrightarrow{D_1 }C^\infty (\Omega, \Lambda^2\mathbb{C}^2 )\longrightarrow 0,\end{split}
\end{equation}
with
\begin{equation}\label{eq:k-CF-}\begin{split}
    & (D_0 \phi)_{A B' }:=\sum_{A'=0',1'}\nabla^{A'}_A\phi_{A'B' }=\nabla^{0'}_A\phi_{0'B' }+\nabla^{1'}_A\phi_{1'B' },\\&
    (D_1 \psi)_{01 }:=\sum_{A'=0',1'}\nabla^{A'}_{ 0}\psi_{1 A' }-\nabla^{A'}_{1}\psi_{ 0A' }=\nabla^{0'}_{0}\psi_{ 1 0' }+\nabla^{1'}_{ 0}\psi_{1 1'
    }-\nabla^{0'}_{ 1}\psi_{ 00' }-\nabla^{1'}_{1}\psi_{0 1' },
\end{split}\end{equation}where  $ A =0,1, B'=0',1'$, $\phi\in C^\infty(\Omega,\odot^{2 }\mathbb{C}^2 )$ has 3 components
$\phi_{0'0'},\phi_{1'0'}=\phi_{0'1'}$ and $\phi_{1'1'}$, while
$D_0 \phi\in C^\infty(\Omega,\mathbb{C}^2\otimes\mathbb{C}^2 )$
has 4 components $(D_0 \phi)_{00'}$,$(D_0 \phi)_{10'}$,
$(D_0 \phi)_{01'}$   and $(D_0 \phi)_{11'}$, and $\Psi=\Psi_{01}\in C^\infty (\Omega, \Lambda^2\mathbb{C}^2 )$ is a  scalar function.

We know from results in \cite{Wa10}  that (\ref{eq:Dirac-Wey}) is a complex: $D_1D_0=0$. It can be checked directly as follows. We calculate, for any
$\phi\in C^\infty(\Omega,\odot^{2 }\mathbb{C}^2 )$,
\begin{equation}\label{eq:exact}\begin{split}
    ( D_1D_0\phi)_{01   } &=\sum_{A'=0',1'} \nabla^{A'}_{ 0}(D_0\phi)_{1 A' }-\nabla^{A'}_{ 1}(D_0\phi)_{ 0 A' }\\&=\sum_{A',C'=0',1'} \nabla^{A'}_{
    0}\nabla^{C'}_1\phi_{C'A' }-\nabla^{A'}_{1}\nabla^{C'}_0\phi_{C'A'  }=0
\end{split}\end{equation}
by $\phi_{C'A'  }=\phi_{A'C'  }$ and the commutativity $\nabla^{A'}_{ 1}\nabla^{C'}_0=\nabla^{C'}_0\nabla^{A'}_{ 1}$,  as scalar differential operators of
constant coefficients.

 The
operator $ D_0 $ in (\ref{eq:Dirac-Wey}) can be written as a
$(4\times 3)$-matrix operator
\begin{equation*}D_0 \phi=\left(\begin{array}{c}(D_0 \phi)_{0 0'}\\(D_0 \phi)_{1 0'}\\(D_0 \phi)_{0 1'}\\(D_0 \phi)_{1 1'}
\end{array}
\right)= \left(\begin{array}{ccc} \nabla_0^{0'}&
\nabla_0^{1'}& 0\\\nabla_1^{0'}& \nabla_1^{1'}& 0\\ 0&\nabla_0^{0'}&
\nabla_0^{1'}\\ 0&\nabla_1^{0'}& \nabla_1^{1'}
\end{array}\right)
\left(\begin{array}{c}\phi_{0'0'}\\\phi_{1'0'}\\\phi_{1'1'}
\end{array}
\right),
\end{equation*}
 and  the
operator $ D_1 $ takes the form
 \begin{equation*}D_1 \psi=  (-\nabla_1^{0'},\nabla_0^{0'},-\nabla_1^{1'},\nabla_0^{1'})
\left(\begin{array}{c}\psi_{0 0'}\\\psi_{1 0'}\\\psi_{0 1'}\\\psi_{1 1'}
\end{array}
\right).
\end{equation*}
Define
\begin{equation*}
     z_0=x_0+ix_1,\qquad z_1=x_2+ix_3
\end{equation*}
and
\begin{equation*}\begin{split}&
     \partial_{z_0}=\partial_{x_0}-i\partial_{x_1},\qquad {\partial}_{\overline z_0}=\partial_{x_0}+i\partial_{x_1},\\&
     \partial_{z_1}=\partial_{x_2}-i\partial_{x_3},\qquad {\partial}_{\overline z_1}=\partial_{x_2}+i\partial_{x_3}.
\end{split}\end{equation*} Our notations coincide with the usual ones up to a factor $\frac 12$.
Using these notations, and the following isomorphisms
\begin{equation*}
 \odot^{2 }\mathbb{C}^2 \cong \mathbb{C}^3,\qquad     \mathbb{C}^2\otimes\mathbb{C}^2 \cong \mathbb{C}^4, \qquad   \Lambda^2\mathbb{C}^2\cong
 \mathbb{C}^1,
\end{equation*}
we can rewrite $D_0: C^\infty (\Omega,  \mathbb{C}^{3} )
 \xrightarrow{  }  C^\infty (\Omega,
  \mathbb{C}^{4 })$ with
\begin{equation*}D_0 \phi =  \left(\begin{array}{rrr} - {\partial}_{\overline z_1}&
-{\partial}_{\overline z_0}& 0\\ \partial_{z_0}& -\partial_{z_1}& 0\\ 0&-{\partial}_{\overline z_1}&
-{\partial}_{\overline z_0}\\ 0&\partial_{{z}_0}& -\partial_{{z}_1}
\end{array}\right)
\left(\begin{array}{c}\phi_{0}\\\phi_{1 }\\\phi_{2}
\end{array}
\right),
\end{equation*}
and  $D_1:    C^\infty (\Omega,
  \mathbb{C}^{4 })\xrightarrow{  }C^\infty (\Omega,  \mathbb{C} )$
 with\begin{equation*}D_1 \psi=  ( -{\partial}_{{z}_0},-{\partial}_{\overline z_1}, {\partial}_{z_1},- {\partial}_{\overline z_0})
\left(\begin{array}{c}\psi_{0 }\\\psi_{1  }\\\psi_{2}\\\psi_{3}
\end{array}
\right).
\end{equation*}
\subsection{The Laplacian associated to 2-Cauchy-Fueter complex}

It is easy to see that
\begin{equation}\label{eq:inner-product}\begin{split}&\overline{\left(\begin{array}{rr } - {\partial}_{\overline z_1}&
-{\partial}_{\overline z_0} \\ \partial_{z_0}& -\partial_{z_1}
\end{array}\right)}^t\left(\begin{array}{rr } - {\partial}_{\overline z_1}&
-{\partial}_{\overline z_0} \\ \partial_{z_0}& -\partial_{ z_1}
\end{array}\right)=\left(\begin{array}{rr } - {\partial}_{z_1}&
{\partial}_{\overline z_0} \\ -\partial_{z_0}& -\partial_{\overline z_1}
\end{array}\right)
\left(\begin{array}{rr } - {\partial}_{\overline z_1}&
-{\partial}_{\overline z_0} \\ \partial_{z_0}& -\partial_{z_1}
\end{array}\right)
=\left(\begin{array}{rr } \triangle &
0 \\ 0& \triangle
\end{array}\right),
 \end{split}\end{equation} where ${}^t$ is the transpose, and
 \begin{equation*}\begin{split}&
\Delta :=\partial_{z_0}{\partial}_{\overline z_0}+\partial_{z_1}\partial_{\overline z_1}=\partial_{x_{0}}^2+ \partial_{x_{ 1}}^2+
 \partial_{x_{2}}^2+ \partial_{x_{ 4}}^2
\end{split}\end{equation*} is the usual Laplacian on $\mathbb{R}^4$.

Let $\mathscr{D}:C^1(\overline{\Omega},\mathbb{C}^{n_1})\rightarrow C^0(\overline{\Omega},\mathbb{C}^{n_2})$ be a differential operator of the first order
with constant coefficients. An operator $\mathscr{D}^* $ is called the
{\it formal adjoint} of $\mathscr{D}  $ if for any $u\in C^1_0( {\Omega},\mathbb{C}^{n_1})$, $ v\in
C^1_0( {\Omega},\mathbb{C}^{n_2})$, we have \begin{equation*}
 \int_\Omega \langle \mathscr{D} u, v\rangle dV=  \int_\Omega \langle u,\mathscr{D}^*
 v\rangle dV,
\end{equation*} where $\langle\cdot,\cdot\rangle$ is the Hermitian inner product in $\mathbb{C}^{n_j}$, $j=1,2$.
It is easy to see that the formal adjoints of $D_0 $ and $D_1 $ are $D_0^*=-\overline{D_0}^t$ and $D_1^*=-\overline{D_1}^t$, respectively. Then,
\begin{equation}\label{eq:D0D0}\begin{split}
    D_0D_0^* &=- \left(\begin{array}{rrr} - {\partial}_{\overline z_1}&
-{\partial}_{\overline z_0}& 0\\ \partial_{z_0}& -\partial_{z_1}& 0\\ 0&-{\partial}_{\overline z_1}&
-{\partial}_{\overline z_0}\\ 0&\partial_{{z}_0}& -\partial_{{z}_1}
\end{array}\right)
 \left(\begin{array}{cccc} - {\partial}_{{z}_1}&\partial_{ \overline z_0}&0&0\\
-\partial_{ z_0}&   - {\partial}_{ \overline{z}_1}&- {\partial}_{{z}_1}&\partial_{ \overline z_0}\\
0&0& - {\partial}_{{z}_0}&-\partial_{ \overline z_1}\end{array}\right) \\&
= -\left(\begin{array}{cccc} \Delta&0 &\partial_{ \overline  z_0} {\partial}_{  {{z}}_1}&-\partial_{ \overline  z_0}^2  \\
 *& \Delta &{\partial}_{ {z}_1}^2&-\partial_{ \overline  z_0} {\partial}_{  {{z}}_1} \\
 *&* &\Delta & 0  \\ *&*&*&\Delta\end{array}\right),
\end{split}\end{equation}
where $* $-entries  are known by Hermitian symmetry of $ D_0^* D_0$, and
\begin{equation}\label{eq:D1D1}\begin{split}
    D_1^* D_1&=-\left(\begin{array}{r}  -{\partial}_{ \overline {z}_0}\\-{\partial}_{  z_1}\\ {\partial}_{ \overline  z_1}\\- {\partial}_{  z_0}
    \end{array}\right) ( -{\partial}_{ {z}_0},-{\partial}_{ \overline z_1}, {\partial}_{  z_1},- {\partial}_{ \overline z_0})\\&  =
    -\left(\begin{array}{cccc} {\partial}_{ z_0}{\partial}_{\overline z_0} & \partial_{  \overline  z_0} {\partial}_{ \overline  {z}_1} &-\partial_{
    \overline  z_0} {\partial}_{ { z}_1} & \partial_{ \overline  z_0}^2 \\ * &    \partial_{ z_1} {\partial}_{\overline { z}_1}& -  {\partial}_{ { z}_1}^2
    &\partial_{ \overline  z_0} {\partial}_{  {{z}}_1}\\ * & * &\partial_{ z_1} {\partial}_{ \overline  { z}_1}& -\partial_{ \overline  z_0} {\partial}_{
    \overline  {z}_1} \\ * &* &* &{\partial}_{ z_0}{\partial}_{\overline z_0}\end{array}\right).
\end{split}\end{equation}
   The sum of (\ref{eq:D0D0}) and (\ref{eq:D1D1})
gives
 \begin{equation}\label{eq:square-1}\begin{split}
 \square_1:=  D_0D_0^*+  D_1^* D_1  &  = -\left(\begin{array}{cccc} \Delta+\partial_{ { z}_0} {\partial}_{ \overline{{z}}_0}&\partial_{ \overline  z_0}
 {\partial}_{ \overline  {z}_1}&0 & 0\\ \partial_{ { z}_0} {\partial}_{ {{z}}_1} & \Delta+ \partial_{ { z}_1} {\partial}_{ \overline{{z}}_1}& 0 &0\\0 &0 &
 \Delta+\partial_{ { z}_1} {\partial}_{ \overline{{z}}_1}& -\partial_{ \overline  z_0} {\partial}_{ \overline  {z}_1} \\0&0&-\partial_{  {z}_0}
 {\partial}_{ {{z}}_1} & \Delta+\partial_{ { z}_0} {\partial}_{ \overline{{z}}_0}\end{array}\right)\\&
   =-\left(\begin{array}{cccc} \Delta+\Delta_1 &    L &0 & 0\\    \overline{ L }& \Delta+ \Delta_2& 0 &0\\0 &0 & \Delta+\Delta_2&   -  L  \\0&0& -
   \overline{ L }& \Delta+\Delta_1\end{array}\right),
\end{split}\end{equation}
where
\begin{equation*}\begin{split}&
\Delta_1:=\partial_{ { z}_0} {\partial}_{ \overline{{z}}_0}=\partial_{x_{0}}^2+ \partial_{x_{ 1}}^2,\\&
\Delta_2:=\partial_{ { z}_1} {\partial}_{ \overline{{z}}_1}=\partial_{x_{2}}^2+ \partial_{x_{ 3}}^2,\\&
    L:=\partial_{ \overline  z_0} {\partial}_{ \overline  {z}_1}=(\partial_{x_{0}}+ {i} \partial_{x_{ 1}})(\partial_{x_{  2}}+ {i}\partial_{x_{ 3}}).
\end{split}\end{equation*}
The operator $\square_1$ is obviously elliptic, i.e., its symbol for any $\xi\neq 0$ is positive definite.
\subsection{Domains of the adjoint operators $D_0^*$ and $D_1^*$} We define the inner product on $L^2( {\Omega},\mathbb{C}^{n })$
by
\begin{equation*}
     (u,v)=\int_\Omega\langle u,v\rangle dV,
\end{equation*}where $\langle\cdot,\cdot\rangle$ is the Hermitian inner product in $\mathbb{C}^{n }$, $dV$ is the Lebesgue measure.

For a   differential operator $\mathscr{D}:C^1(\overline{\Omega},\mathbb{C}^{n_1})$ $\rightarrow C^0(\overline{\Omega},\mathbb{C}^{n_2})$ of the first
order with constant coefficients,
$u\in C^1(\overline{\Omega},\mathbb{C}^{n_1})$ and $ v\in
C^1(\overline{\Omega},\mathbb{C}^{n_2})$, we have
\begin{equation}
 \int_\Omega \langle \mathscr{D} u, v\rangle dV=\int_\Omega \langle u,\mathscr{D}^*
 v\rangle dV +\int_{\partial\Omega}
 \langle  u, \mathscr{D}^*(\nu)v\rangle dS,\label{eq:green00}
\end{equation}
by Green's formula, where $\nu=(\nu_0,\ldots,\nu_4)$ is the unit vector of outer normal to the boundary, and  $\mathscr{D}^*(\nu)$ is obtained by
replacing $\partial_{x_j}$ in $\mathscr{D}^* $ by $\nu_j$.

  By abuse of notations, we denote also by $\mathscr{D}^*$ the adjoint operator of $\mathscr{D}:L^2( {\Omega},\mathbb{C}^{n_1}) \rightarrow L^2(
  {\Omega},\mathbb{C}^{n_2})$.
Now let $\Omega$   be $\mathbb{R}^4_+=\{x=(x_0,\ldots,x_3)\in \mathbb{R}^4;x_0>0\}$. Then the unit inner normal vector is $\nu=(1,0,0,0)$. By definition
of the adjoint operator, a function $\psi=(\psi_{0  },\psi_{1 },\psi_{2},\psi_{3})^t$ $\in {\rm Dom} D_0^*\cap C^1 (\Omega,
 \mathbb{C}^4)$ if and only if the integral over the boundary in (\ref{eq:green00}) vanishes for any $u$, i.e., $ D_0^*(\nu)\psi=0$ on the boundary. Then,
 \begin{equation*}
  0=  \left(\begin{array}{cccc} -  {\partial}_{{z}_1}&\partial_{ \overline z_0}&0&0\\
-\partial_{ z_0}&   - {\partial}_{\overline {z}_1}&- {\partial}_{{z}_1}&\partial_{\overline  z_0}\\
0&0& - {\partial}_{{z}_0}&-\partial_{\overline  z_1}\end{array}\right) (\nu)\psi|_{\partial\Omega}= \left(\begin{array}{cccc} 0&1&0&0\\-1&  0&0&1\\
0&0&-1 & 0\end{array}\right) \psi|_{\partial\Omega},
 \end{equation*}
 from which we get
 \begin{equation}\label{eq:boundary1}
     \psi_{1}=\psi_{2}=0,\qquad \psi_{0}-\psi_{3}=0\qquad {\rm on} \quad \partial\Omega.
 \end{equation}

 Similarly, $\Psi  \in {\rm Dom} D_1^*\cap C^1 (\Omega,
  \mathbb{C} )$ if and only if $ D_1^*(\nu)\Psi=0$ on the boundary, i.e.,
 \begin{equation*}
  0=   \left(\begin{array}{r}  -{\partial}_{\overline {z}_0}\\-{\partial}_{z_1}\\ {\partial}_{ \overline z_1}\\- {\partial}_{ z_0} \end{array}\right)
  (\nu)\Psi|_{\partial\Omega}= \left(\begin{array}{r}-1 \\0 \\0  \\ -1 \end{array}\right)  \Psi|_{\partial\Omega},
 \end{equation*}
 from which we get $ \Psi|_{\partial\Omega}=0$. Now $D_1\psi\in {\rm Dom} D_1^*\cap C^1 (\Omega,
  \mathbb{C}  )$ implies that
\begin{equation*} -{\partial}_{{z}_0}\psi_{0  }-{\partial}_{\overline z_1}\psi_{1}+ {\partial}_{ z_1}\psi_{2}- {\partial}_{\overline z_0}\psi_{3}
=0,\qquad {\rm on}\quad \partial \Omega.
\end{equation*}
Note that ${\partial}_{\overline z_1}\psi_{1}={\partial}_{ z_1}\psi_{2}=0$ since ${\partial}_{\overline z_1}$ and ${\partial}_{ z_1} $ are tangential
derivatives, and $\psi_{1 }$, $\psi_{2} $ both vanish on the boundary by using (\ref{eq:boundary1}). Therefore,
\begin{equation}\label{eq:boundary2}
  {\partial}_{{z}_0}\psi_{0  }+ {\partial}_{\overline z_0}\psi_{3} = \partial_{x_0}(  \psi_{0 }+\psi_{3})=0,\qquad {\rm on}\quad \partial \Omega
\end{equation}
by using (\ref{eq:boundary1}) again. So we need to solve the system  $\square_1^{(2)}    {\psi} =f$ in $\Omega$ under the boundary conditions
(\ref{eq:boundary1}) and (\ref{eq:boundary2}).

We  need to define more operators. We obtain $\square_0:=   D_0^* D_0 $ equals to
\begin{equation}\label{eq:square0}\begin{split}
&  -
 \left(\begin{array}{cccc} - {\partial}_{{z}_1}&\partial_{ \overline z_0}&0&0\\
-\partial_{ z_0}&   - {\partial}_{ \overline{z}_1}&- {\partial}_{{z}_1}&\partial_{ \overline z_0}\\
0&0& - {\partial}_{{z}_0}&-\partial_{ \overline z_1}\end{array}\right)\left(\begin{array}{rrr} - {\partial}_{\overline z_1}&
-{\partial}_{\overline z_0}& 0\\ \partial_{z_0}& -\partial_{z_1}& 0\\ 0&-{\partial}_{\overline z_1}&
-{\partial}_{\overline z_0}\\ 0&\partial_{{z}_0}& -\partial_{{z}_1}
\end{array}\right)
= -\left(\begin{array}{ccc } \Delta&0 &0\\
0 & 2\Delta &0\\
0 & 0&\Delta \end{array}\right),
\end{split}\end{equation}
and
 \begin{equation}\label{eq:square2}\begin{split}
 \square_2:= D_1  D_1^* &=-( -{\partial}_{{z}_0},-{\partial}_{\overline z_1}, {\partial}_{z_1},- {\partial}_{\overline z_0}) \left(\begin{array}{r}
 -{\partial}_{\overline {z}_0}\\-{\partial}_{z_1}\\ {\partial}_{ \overline z_1}\\- {\partial}_{z_0} \end{array}\right) =-2\Delta,
\end{split}\end{equation}
 with the boundary condition $\Psi\in {\rm Dom} D_1^*\cap C^1 (\Omega,
  \mathbb{C}^{1})$, i.e., the Dirichlet condition $\Psi|_{\partial\Omega}=0$.

\begin{cor}\label{cor:green0} Suppose that $u \in H^1( {\Omega}, \mathbb{C}^{n_1}) $, $ v\in H^1( {\Omega}, \mathbb{C}^{n_2}) $, and
$\mathscr{D}(\nu)u|_{\partial\Omega}=0$ or $\mathscr{D}^*(\nu)v|_{\partial\Omega}$ $=0$. Then
\begin{equation} \label{eq:green0}
    (\mathscr{D}u,v)= (u,\mathscr{D}^* v),\qquad (v, \mathscr{D}u )= ( \mathscr{D}^*v,u)
\end{equation}
     \end{cor}
\begin{proof}
     The trace theorem states that the operator of restriction to the boundary $  H^s(\Omega, \mathbb{C}^n) \rightarrow H^{s-\frac 12}(\partial\Omega,
     \mathbb{C}^n)$ for $s>\frac 12 $ is a bounded operator (cf. Proposition 4.5 in chapter 4 in \cite{Ta}). Moreover, $C^\infty(\overline{\Omega},
     \mathbb{C}^n)$ is dense in $H^s(\Omega, \mathbb{C}^n)$ for $s\geq 0$.   Approximating $u \in H^1( {\Omega}, \mathbb{C}^{n_1}) $, $ v\in H^1(
     {\Omega}, \mathbb{C}^{n_2}) $, by functions from $C^\infty(\overline{\Omega}, \mathbb{C}^{n_j})$, we see that integration by part (\ref{eq:green00})
     holds for $u \in H^1( {\Omega}, \mathbb{C}^{n_1}) $, $ v\in H^1( {\Omega}, \mathbb{C}^{n_2}) $, (cf. (7.2) in chapter 5 in \cite{Ta}).  The boundary
     term vanishes by the assumption.
\end{proof}
\section{The Shapiro-Lopatinskii condition}
\subsection{Definition of the Shapiro-Lopatinskii condition}
Assume that $P(x,\partial):C^\infty(\overline{\Omega},E_0)\rightarrow C^\infty(\overline{\Omega},E_1)$ is an elliptic differential operator of order $m$,
and that
$B_j(x,\partial):C^\infty(\overline{\Omega},E_0)\rightarrow C^\infty(\partial{\Omega},G_j)$, $j=1,\ldots,l$,  are differential operators of order $m_j\leq
m-1$, where $E_0,E_1, G_j$, $j=1,\ldots,l$,  are finite dimensional complex vector spaces. Let $\Omega$ be a
domain in $ \mathbb{R}^n$ with smooth boundary $\partial\Omega$.
Consider the  boundary value problem
\begin{equation} \left\{\begin{array}{l} P(x,\partial)u=f,\qquad
{\rm on }\quad \Omega,\\
B_j(x,\partial)u=g_j,\quad {\rm on }\quad \partial\Omega,\quad
j=1,\ldots, l.
\end{array}
\right.\label{eq:bvp-general}
\end{equation}

For fixed $x\in \partial \Omega$, define the
half space $V_x:=\{y\in\mathbb{R}^n; \langle y,\nu_x\rangle>0\}$,
where $\nu_x$ is the unit vector of inner normal to $ \partial \Omega$ at
point $x$. By  a rotation if necessary, we can assume
$n_x=(1,0,\ldots, 0 )$ and $P(x,\partial_x)$ can be written as
\begin{equation}
P(x,\partial )=\frac {\partial^m}{\partial x_1^m}+\sum_{\alpha=0}^{m-1}
A_\alpha(x,\partial_{x'}) \frac  {\partial^\alpha}{\partial
x_1^\alpha},\label{eq:PD}
\end{equation}
up to multiply an invertible matrix function,   where the order of
$A_\alpha(x,\partial_{x'})$ is equal to $m-\alpha$, $x'=(x_2,\ldots,x_{ n })$.
 For the
elliptic operator $P(x,\partial)$, the boundary value problem
(\ref{eq:bvp}) is called {\it regular} if  for any $\xi\in
\mathbb{R}^{n-1}$ and $\eta_j\in G_j$, there is a unique bounded  solution
on $\mathbb{R}_+=[0,\infty)$ to the Cauchy problem
\begin{equation}
 \frac {d^m\Phi}{d t^m}+\sum_{\alpha=0}^{m-1} \widetilde{A}_\alpha(\xi)
\frac  {d^\alpha\Phi}{dt^\alpha}=0,\quad \widetilde{B}_j\left(\xi,\frac d{dt}\right)\Phi(0)=\eta_j,
\quad j=1,\ldots, l,\label{eq:Lopatinski-Shapiro0}
\end{equation}
Here $\Phi$ is a $E_0$-valued function over $\mathbb{R}_+$,
  $\widetilde{A}_\alpha(\xi)$   is the homogeneous part of $
{A}_\alpha(x,\xi)$ of   degree $m-\alpha$, and $ {A}_\alpha(x,\xi)$ is obtained by
replacing  $\partial_{y'}$ in $A_\alpha(x,\partial_{y'})$ by
$ i\xi$ (this condition is the same if it is replaced by $\frac 1i  \xi$ ). The operator $\widetilde{B}_j(\xi,d/dt)$ is defined similarly. The regularity
property
  is equivalent to the fact that there is no   nonzero bounded solution
on $\mathbb{R}_+$ to the Cauchy problem
\begin{equation}
 \frac {d^m\Phi}{d t^m}+\sum_{\alpha=0}^{m-1} \widetilde{A}_\alpha(\xi)
\frac  {d^\alpha\Phi}{dt^\alpha}=0,\quad \widetilde{B}_j\left(\xi,\frac d{dt}\right)\Phi(0)=0,
\quad j=1,\ldots, l.\label{eq:Lopatinski-Shapiro}
\end{equation}
Furthermore, it is equivalent to the fact that there is no   nonzero rapidly decreasing   solution
on $\mathbb{R}_+$ to the Cauchy problem (\ref{eq:Lopatinski-Shapiro})
 (cf.
 ($ii'$) in p. 454 in \cite{Ta}).
This condition
is usually called the {\it Shapiro-Lopatinskii condition}.

The latter condition can also be stated without using rotations (cf. \S 20.1.1 in \cite{Hor} and the discussion below it).
For $x\in \partial\Omega$, and $\xi \perp \nu_x$,
the map
\begin{equation}\label{eq:LS'}
     M_{x,\xi}\ni u\longrightarrow (B_1(x,i\xi+\nu_x\partial_{t})u(0),\ldots,B_l(x,i\xi+\nu_x\partial_{t})u(0))
\end{equation}
is bijective, where $M_{x,\xi}$ is the set of all solutions $u\in C^\infty( \mathbb{R}_+,E_0)$ satisfying
 \begin{equation}\label{eq:P-t}
     P(x,i\xi+\nu_x\partial_{t})u(t)=0
 \end{equation}
which are bounded on $\mathbb{R}_+$.  Here for a differential operator $P$, the notation $P(\xi +\nu\partial_t)$
means that $ \partial_{x_j}$ is replaced by  $ i\xi_j+\nu_j\partial_t$, $j=1,\ldots,n$. Equivalently, there is no   nonzero rapidly decreasing   solution
on $\mathbb{R}_+$ to the ODE (\ref{eq:P-t}) under the initial condition
\begin{equation}\label{eq:Lopatinski-Shapiro-H}
      B_j(x,i\xi+\nu_x\partial_{t})u(0)=0,\qquad j=1,\ldots, l.
\end{equation}

\subsection{Checking the  Shapiro-Lopatinskii condition for $k=2$}

 \begin{prop} Suppose $\Omega$ is a smooth domain in $\mathbb{R}^4 $. The boundary value problem
\begin{equation}\label{eq:bvp-K}
  \left\{\begin{array}{l}  (D_0D_0^*+  D_1^* D_1)\psi=0,\qquad {\rm on }\quad \Omega,\\
D_0^*(\nu) \psi|_{\partial \Omega}=0,\\ D_1^*(\nu)D_1 \psi|_{\partial \Omega}=0,
 \end{array} \right.
\end{equation}
is regular.\label{prop:regular-bvp-K}
\end{prop}
\begin{proof} Here we check the Lopatinski-Shapiro condition by
generalizing the method proposed by Dain in \cite{D06-2}, which we have used in \cite{W08}. Originally, this method works for operator of type $K^*K$ for
some differential operator $K$ of first order, while here our operator has the form $D_0D_0^*+  D_1^* D_1$.

Fix a point in the boundary $  \partial\Omega$. Without loss of generality, we assume this point to be the origin. Denote by $\nu\in \mathbb{R}^4$ the
unit vector of inner normal to the boundary  at the origin. Let
\begin{equation*}
     {\mathscr V}_\nu=\{x\in \mathbb{R}^4; x\cdot\nu>0\}  \end{equation*}
     be a half-space. For any fixed vector $\xi \perp \nu $,
suppose that $u(t )$ is a rapidly decreasing solution
on $[0,\infty)$ to the following ODE   under the initial condition:
   \begin{equation}\label{eq:Lopatinski-Shapiro-K}\left\{\begin{split}&
(D_0D_0^*+  D_1^* D_1)(i\xi +\nu \partial_t)u( t)=0,\\& D_0^*(\nu)u(0)=0,\\&
D_1^*(\nu)D_1(i\xi +\nu\partial_t)u(0)=0.\end{split}\right.\end{equation}  Let us prove that $u$ vanishes.
Now define a
function $U:{\mathscr V}_\nu\rightarrow  \mathbb{ C}^4$ by
\begin{equation} U(x)=e^{i  x \cdot \xi  }u( x\cdot \nu)\label{eq:U}
\end{equation}
for $x \in {\mathscr V}_\nu$. Note that for a differential operator $Q=\sum_{j=0}^3Q_j\partial_{x_j}$, where the $Q_j$'s are $(4\times 3)$-matrices, we
have
$Q U(x)
=     \sum_{j=0}^3Q_j\left(i\xi_j u( x\cdot \nu)+\nu_ju'( x\cdot \nu)\right)e^{i  x \cdot \xi  }.
$
Then it is easy to see that (\ref{eq:Lopatinski-Shapiro-K})
 implies
\begin{equation} \label{eq:Lopatinski-Shapiro-K2}
\left\{\begin{array}{l}  (D_0D_0^*+  D_1^* D_1)U(x)=0,\qquad {\rm on }\quad {\mathscr V}_\nu,\\
D_0^*(\nu)U(x)|_{\partial {\mathscr V}_\nu}=0,\\ D_1^*(\nu)D_1 U(x)|_{\partial {\mathscr V}_\nu}=0,
 \end{array} \right.
\end{equation}
It is sufficient to show that $U$ vanishes. Consider the interval
$I_\xi=\{s\xi\in \partial {\mathscr V}_\nu ;   |s|\leq \frac \pi{|\xi| }\}$, the ball
$B_\xi= \{y'\in \partial {\mathscr V}_\nu; y'\perp\xi, |y'|\leq r\} $ for any fixed $r>0$, and the
domain
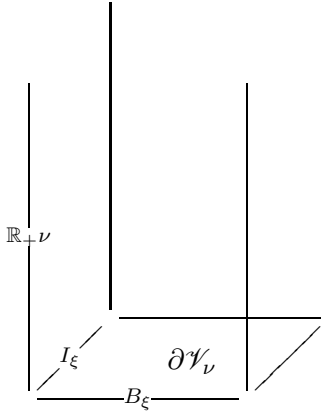
\begin{equation}\mathscr{D}_\xi= I_\xi\times B_\xi \times \mathbb{R}_+\nu ,
\end{equation}
where $\mathbb{R}_+\nu =\{t\nu; t\in \mathbb{R}_+\}$.
\begin{equation*}
     \xy 0;/r.17pc/:
(0,0 )*+{ } ="1";
(40, 0)*+{ } ="2";
( 15,15)*+{ } ="3";
(55,15)* +{ }="4";
(0,60 )*+{ }="5";
(40, 60)* +{ }="6";
(  15,75)*+{ } ="7";
(55,75)* +{ }="8";(30,7 )*+{\partial {\mathscr V}_\nu } ="9";
{\ar@{-}|-{
  B_\xi } "1";"2" };
{\ar@{-}|-{ I_\xi  } "3";"1" };
{\ar@{-}|-{  \mathbb{R}_+\nu } "5";"1" };
{\ar@{-}|-{}  "4";"2" };
{\ar@{-}|-{}  "6";"2" };
{\ar@{-}|-{} "7";"3" };
{\ar@{-}|-{}  "3";"4" };
{\ar@{-}|-{}  "8";"4" };
 \endxy
\end{equation*}

Since $U$ in (\ref{eq:U})  rapidly decays in direction $\nu$,
by Green's formula (\ref{eq:green00}), we have
\begin{equation} \begin{split}\int_{\mathscr{D}_\xi}&\langle (D_0D_0^*+  D_1^* D_1) U,
U\rangle=\int_{\mathscr{D}_\xi} \langle   D_0^*  U,D_0^*
 U\rangle+\int_{\mathscr{D}_\xi} \langle   D_1 U,  D_1
 U\rangle \\&
 -\int_{I_\xi\times B_\xi\times\{0\} \cup \partial
I_\xi\times B_\xi\times \mathbb{R}_+ \nu \cup I_\xi\times
\partial B_\xi\times \mathbb{R}_+ \nu}(\langle D_0^* U,D_0^*(\nu )
U\rangle-\langle D_1^*(\nu )D_1 U,
U\rangle) dS,
\end{split}\label{eq:Green}\end{equation}
where  $\langle\cdot,\cdot\rangle$ is the standard Hermitian inner product in $\mathbb{C}^4$.

(1)
The
integral $\int_{I_\xi\times B_\xi\times\{0\}}$ in (\ref{eq:Green})
vanishes by the boundary condition $D_0^*(\nu )
U=0$ and $ D_1^*(\nu )D_1 U=0$ on $\partial {\mathscr V}_\nu
$ in (\ref{eq:Lopatinski-Shapiro-K2});

(2) The integral
 $\int_{\partial I_\xi\times
B_\xi\times \mathbb{R}_+ \nu}$ vanishes since $U$, $D_0^*U$ and $ D_1 U$ are periodic in
direction $\xi$,  and on the opposite surface, we have the identity
$D_j^*(\nu )|_{\{\xi\}\times B_\xi\times \mathbb{R}_+\nu}=-D_j^*(\nu
)|_{\{-\xi\}\times B_\xi\times \mathbb{R}_+\nu}$, $j=0,1$;

(3) Similarly, the integral $\int_{
I_\xi\times
\partial B_\xi\times \mathbb{R}_+ \nu}$
vanishes since $U$, $D_0^*U$ and $ D_1 U$ are constant in any direction in $ B_\xi$,
and on the opposite direction, we have the identity
$D_j^*(\nu )|_{I_\xi\times \{v\}\times \mathbb{R}_+\nu}=-D_j^*(\nu
)|_{I_\xi\times\{-v\} \times \mathbb{R}_+\nu}$ for any $v \in B_\xi$.

Obviously, the integral in the left hand side of (\ref{eq:Green}) vanishes by the first equation in (\ref{eq:Lopatinski-Shapiro-K2}). Consequently,
\begin{equation*}
\int_{\mathscr{D}_\xi} \langle   D_0^*  U,D_0^*
 U\rangle+ \langle   D_1 U,  D_1
 U\rangle=0,
\end{equation*}
i.e.,
\begin{equation}\label{eq:harmonic}
     D_0^*
 U=0, \qquad  D_1 U=0,\qquad {\rm on }\quad {\mathscr V}_\nu.
\end{equation}

By applying the following Proposition \ref{cor:exact} to the convex domain $  {\mathscr V}_\nu$,  we see that
  there exists a function $\widetilde{U}\in C^\infty(  {\mathscr V}_\nu , \mathbb{C}^{3} )$ such that $D_0\widetilde{U}=U$ on $  {\mathscr V}_\nu$, and so
  $ D_0^*D_0\widetilde{U}=0$ by the first identity in (\ref{eq:harmonic}). By the explicit form of $ D_0^*D_0 $ in (\ref {eq:square0}), we see that each
  component of $\widetilde{U} $ is harmonic on $  {\mathscr V}_\nu$. Consequently, each component of $U= D_0\widetilde{U}$ is also harmonic on $
  {\mathscr V}_\nu$ since $\Delta U= \Delta D_0\widetilde{U}=D_0\Delta\widetilde{U}=0$ by $  D_0 $ being a   differential operator of constant
  coefficients and $\triangle$ being a scalar differential operator of constant coefficients. This
  implies that
  \begin{equation}\label{eq:harmonic-BVP}\left\{\begin{array}{l}  \Delta U_{0  }=\Delta U_{1  }=\Delta U_{2}=\Delta U_{3}=0,\qquad {\rm on }\quad
  {\mathscr V}_\nu,\\
D_0^*(\nu) U|_{\partial {\mathscr V}_\nu}=0,\\ D_1^*(\nu)D_1 U|_{\partial {\mathscr V}_\nu}=0.
 \end{array} \right.
   \end{equation}
 In particular, when $\nu=(1,0,0,0)$, we have
  \begin{equation}
\left\{\begin{array}{l} \Delta U_{0  }=\Delta U_{1 }=\Delta U_{2}=\Delta U_{3}=0,\qquad {\rm on }\quad \mathbb{R}^{4 }_+,\\
 U_{1}|_{\mathbb{R}^3}=U_{2}|_{\mathbb{R}^3}=0,\\ (U_{0} -U_{3})|_{\mathbb{R}^3}=0,\\ \partial_{x_0}(  U_{0}+U_{3})|_{\mathbb{R}^3}=0,
 \end{array} \right.
  \end{equation} by the boundary conditions  (\ref{eq:boundary1})-(\ref{eq:boundary2}) for the upper half-space.
Note that a harmonic function on $\mathbb{R}_+^4$ with vanishing boundary value must vanish. We see that  $U_{1}\equiv U_{2}\equiv U_{0} -U_{3}\equiv 0$
and $ \partial_{x_0}(  U_{0}+U_{3}) \equiv0$. Consequently,  $    U_{0}+U_{3}  $ is independent of $x_0$, and so vanishes since it is  rapidly decreasing
in $x_0$. Therefore,
$U\equiv 0$.

For the general case of $\nu$, we
set
\begin{equation}\label{eq:zeta-nu}
     \zeta_0=\nu_0-{i}\nu_1,\qquad  \zeta_1=\nu_2- {i}\nu_3.
\end{equation}
Then
\begin{equation}\label{eq:D0-nu}D_0(\nu) = \left(\begin{array}{rrr} -\overline{\zeta_1}&
-\overline{\zeta_0}& 0\\  {\zeta_0}& - {\zeta_1}& 0\\ 0&  -\overline{\zeta_1}&
-\overline{\zeta_0}\\ 0&  {\zeta_0}& - {\zeta_1}
\end{array}\right) ,
\end{equation}
and
\begin{equation}\label{eq:D1-nu}D_1(\nu) =  (- {\zeta_0},- \overline{\zeta_1},   {\zeta_1},
-  \overline{ \zeta_0})
 .
\end{equation}
It is direct to check that $D_1(\nu)D_0(\nu)=0$, that also follows from $D_1 D_0 =0$.
Note that
\begin{equation}\label{eq:det}\det \left(\begin{array}{rrr} -\overline{\zeta_1}&
-\overline{\zeta_0 }\\  {\zeta_0}& - {\zeta_1}&
\end{array}\right)=|\zeta_0|^2+| \zeta_1|^2,
\end{equation}and therefore $D_0(\nu)$ in (\ref{eq:D0-nu}) has rank $3$. The vector $
D_1(\nu)$ in (\ref{eq:D1-nu}) does not vanish for nonvanishing $\nu $, i.e., $
D_1(\nu)$ has rank $1$. Hence,
${\rm Im}D_0(\nu)=\ker D_1(\nu)$ and ${\rm Im}D_1(\nu)^*$ is a $1$-dimensional   space orthogonal to $\ker D_1(\nu)$.
Namely we have an exact sequence
 \begin{equation*}
    0  \rightarrow\mathbb{C}^3\xrightarrow{
D_0(\nu) }\mathbb{C}^4 \xrightarrow{D_1 (\nu) }\mathbb{C}^1\rightarrow 0
 \end{equation*}(cf. Lemma 3.1 in the following),  and the orthogonal  decomposition
\begin{equation}\label{eq:C4=U'+U''}
   \mathbb{C}^4={\rm Im}D_0(\nu)\oplus{\rm Im}D_1(\nu)^*,
\end{equation}(cf.  (2.13) in \cite{Wa08'} for decompositions of such type).
We rewrite $U$ as
\begin{equation*}
    U= D_0(\nu)U'+ D_1(\nu)^*U'',
\end{equation*} for some $ \mathbb{C}^3$-valued function $U'$ and scalar  function $U''$. Such $U'$ and $U''$ are unique.
 Then,
\begin{equation}\label{eq:U''}
     D_0^*(\nu) U= D_0^*(\nu) ( D_0(\nu)U'+ D_1(\nu)^*U'')=D_0^*(\nu) D_0(\nu)U' .
\end{equation}
Here $ D_0^*(\nu) D_0(\nu)$
is an invertible $(3\times3)$-matrix because $ D_0(\nu)$ has rank $3$. It follows from $D_0^*(\nu) D_0(\nu)\triangle U'=D_0^*(\nu)  \triangle U=0 $ that
$U'$ is harmonic. The second equation in (\ref{eq:harmonic-BVP})
together with (\ref{eq:U''})
implies that $U' =0$ on the boundary $\partial {\mathscr V}_\nu$, and so it vanishes as a harmonic function  on the whole half space ${\mathscr V}_\nu$.
Now we have $ U=D_1(\nu)^*U''$ (we must have $U''=\frac 12 D_1(\nu) U$). $U''$ is also  harmonic.

   The third  equation in (\ref{eq:harmonic-BVP})
implies that the scalar function $D_1 U|_{\partial {\mathscr V}_\nu}=0$.   Then,
\begin{equation}\label{eq:nu-der}
\begin{split}
D_1 U&=D_1 D_1(\nu)^*U''= \left ( -{\partial}_{ {z}_0},-{\partial}_{ \overline z_1}, {\partial}_{  z_1},- {\partial}_{ \overline z_0}\right)D_1(\nu)^*U''
\\&=\left (-(\partial_{x_{0}}- {i} \partial_{x_{ 1}} ) , -(\partial_{x_{ 2}}+
{i}\partial_{ x_{ 3}}),   \partial_{x_{  2}}- {i}\partial_{x_{ 3}}  , -(\partial_{ x_{0 }}+ {i}\partial_{  x_{ 1}})  \right)
\left(\begin{array}{r } -(\nu_0+{i}\nu_1)\\-(\nu_2- {i}\nu_3 )\\ \nu_2+ {i}\nu_3 \\
-(\nu_0-{i}\nu_1 ) \end{array}\right)U''  \\&
=2(\nu_0 \partial_{x_{0}}+ \nu_1 \partial_{x_{ 1}}+ \nu_2\partial_{x_{ 2}}+
 \nu_3\partial_{ x_{ 3}}) {U}''=2  \partial_{\nu  }{U}'' = 0
\end{split}
\end{equation}
on the boundary $\partial {\mathscr V}_\nu$. As $U''$ is a harmonic function, we must have $\partial_\nu{U}'' \equiv 0$ on the whole half space $
{\mathscr V}_\nu$. So  ${U}'' $ is constant in the direction $\nu$. But it is also rapidly decreasing along this direction. Hence
$  {U}'' \equiv 0$ on $  {\mathscr V}_\nu$. Thus $   {U} $ vanishes on $  {\mathscr V}_\nu$.
\end{proof}
\subsection{The solvability of the  non-homogeneous  $k$-Cauchy-Fueter  equations on convex domains without estimate}
The following proposition is proved in \cite{Wa10} for any dimension by using twistor transformations. Here we give an elementary proof.
\begin{prop} \label{cor:exact} The  sequence
\begin{equation}\label{eq:Dirac-Wey-convex} \begin{split}C^\infty (\Omega,  \mathbb{C}^3 )
 \xrightarrow{D_0 }  C^\infty (\Omega,
 \mathbb{C}^{4 })\xrightarrow{D_1 }C^\infty (\Omega,  \mathbb{C}^1 ),\end{split}
\end{equation}is exact for any convex domain $\Omega$. Namely, for any $\psi\in  C^\infty (\Omega,
 \mathbb{C}^{4 })$  satisfying $D_1\psi=0$, there exists $\phi\in C^\infty (\Omega,
 \mathbb{C}^{3 })$ such that
\begin{equation*}
     D_0\phi=\psi \qquad {\rm on}\quad \Omega.
\end{equation*}
\end{prop}

Let $\mathscr{E} (\Omega) $ be the set of $C^\infty$ functions
on $\Omega$ and let   $\mathscr{R}$   be the ring
 of polynomials  $\mathbb{C}[\xi_0,\xi_1,\ldots,\xi_{n}]$. For a  positive integer $p$, $ \mathscr{R}^p$ denotes the space of all vectors $(f_1,\ldots,f_p
 )^t$ with $f_1,\ldots,f_p \in \mathscr{R}$, and  $\mathscr{E}^p(\Omega)$ is
 defined similarly.
 The following result is essentially due
 to Ehrenpreis-Malgrange-Palamodov.
\begin{thm} \label{thm:equivalent} {\rm (cf. Theorem A in \cite{Na})} Let $A(\xi)$, $B(\xi)$ be respectively $(q\times p)$ and
$(r\times q)$ matrices of polynomials, and let  $A(D)$ and $ B(D)$
be differential operators obtained by substituting
$\partial_{x_j}$ to $\frac 1i\xi_j$ to $A(\xi)$ and $B(\xi)$,
respectively. Then the following statements are equivalent:
\begin{itemize}
\item[(1)] the sequence $ \mathscr{R}^p\xleftarrow{
A(\xi)^t}\mathscr{R}^q \xleftarrow{   B(\xi)^t}\mathscr{R}^r$ is
exact,
\item[(2)] the sequence $ \mathscr{E}^p(\Omega)\xrightarrow{  A(D)}
\mathscr{E}^q(\Omega) \xrightarrow{ B(D)} \mathscr{E}^r(\Omega) $
is exact for any convex and non empty domain $\Omega\subset \mathbb{R}^{n+1}$.
\end{itemize}
\end{thm}

\begin{lem} \label{lem:exact} The  sequence
\begin{equation*}
  0  \leftarrow \mathbb{C}^3\xleftarrow{
D_0(\xi)^t}\mathbb{C}^4 \xleftarrow{D_1 (\xi)^t}\mathbb{C}^1\leftarrow 0
\end{equation*}
 is
exact for any nonzero $\xi\in \mathbb{C}^4 $.
\end{lem}
\begin{proof} Set \begin{equation}\label{eq:eta}
     \eta_0=\xi_0- {i}\xi_1,\qquad  \eta_1=\xi_2- {i}\xi_3.
\end{equation}
Then
\begin{equation}\label{eq:D0-xi}D_0(\xi)^t =\frac 1i \left(\begin{array}{rrrr} -\overline{\eta_1}&  {\eta_0}&0& 0\\
-\overline{\eta_0}& - {\eta_1}& -\overline{\eta_1}&  {\eta_0} \\ 0&  0&
-\overline{\eta_0}&- {\eta_1}
\end{array}\right) ,
\end{equation}
and
\begin{equation}\label{eq:D1-xi}D_1(\xi)^t = \frac 1i \left(\begin{array}{r}  - {\eta_0}\\- \overline{\eta_1}\\  {\eta_1}\\
- \overline{\eta_0}\end{array}\right)
 .
\end{equation}
The proof of ${\rm Im}D_1(\xi)^t=\ker D_0(\xi)^t$ is similar to the paragraph  below
(\ref{eq:det}).
\end{proof}

\begin{prop}\label{prop:exact}  The sequence $ \mathscr{R}^3\xleftarrow{
 D_0 (\xi)^t}\mathscr{R}^4 \xleftarrow{ D_1 (\xi)^t}\mathscr{R}^1$ is
exact.
\end{prop}
\begin{proof} It is obvious that $
 D_0 (\xi) ^t  D_1 (\xi)^t=0$  by (\ref{eq:D0-xi})-(\ref{eq:D1-xi}). Suppose $
 D_0 (\xi)^t\left(\begin{array}{c } p_1(\xi)\\\vdots\\ p_4(\xi) \end{array}\right)=0$, where $p_j$ are polynomials. By   Lemma \ref{lem:exact}, for each
 $\xi\neq 0$, there exists an element of $\mathbb{C}^1$, say $f_\xi $, such  that
\begin{equation*}
     \left(\begin{array}{c } p_1(\xi)\\\vdots\\ p_4(\xi) \end{array}\right)= D_1 (\xi)^t f_\xi= \frac 1i \left(\begin{array}{r } -  \xi_0+ i\xi_1\\ -
     \xi_2-i\xi_3\\ \xi_2-i\xi_3\\- \xi_0-i \xi_1 \end{array}\right)f_\xi.
\end{equation*} It follows from the first two equations that
 $ (\xi_0+ i\xi_1)p_1(\xi)+(\xi_2-i\xi_3) p_2(\xi)=i(\xi_1^2+\ldots+\xi_4^2)f_{\xi }$ on $\mathbb{R}^4\setminus\{0\}$. Then $f_{\xi }$ is a rational
 function $Q(\xi)/(\xi_1^2+\ldots+\xi_4^2)$ for some polynomial $Q(\xi)$. The first equation above implies   the following identity of polynomials:
 \begin{equation*}
    i  p_1(\xi)(\xi_1^2+\ldots+\xi_4^2)=(-  \xi_0+ i\xi_1) Q(\xi).
 \end{equation*} This equation also holds on $\mathbb{C}^4$ by natural extension of polynomials. By comparison of zero loci, we see that
  $ -  \xi_0+ i\xi_1 $ must be a factor of $ p_1(\xi)$. Namely, $p_1(\xi)=(-  \xi_0+ i\xi_1) q(\xi)$ for some polynomial $q(\xi)$. Consequently, $f_\xi=i
  q(\xi)$ is  a polynomial on $\mathbb{R}^4$. The  result follows.
\end{proof}
Applying Theorem \ref{thm:equivalent} to the exact sequence in Proposition \ref{prop:exact}, we get the Proposition \ref{cor:exact}.

\section{The case $  k>2$}
\subsection{The   operators $D_0^{(k)} $ and $D_1^{(k)} $ and the associated Laplacian}

The operators in the $k$-Cauchy-Fueter complex (\ref{eq:Dirac-Wey-k}) are given by
(\ref{eq:Dirac-Wey-k-0}).
If we use notations
\begin{equation}\label{eq:Dk}
    \phi=\left(\begin{array}{c}\phi_{0'0'\ldots 0'0'}\\\phi_{1'0'\ldots 0'0'}
\\\vdots\\\phi_{1'1'\ldots 1' 0'}\\\phi_{1'1'\ldots 1'1'}
\end{array}
\right)=\left(\begin{array}{c}\phi_{0 }\\\phi_{1 }
\\\vdots\\\phi_{k }
\end{array}
\right),\qquad  \psi:=\left(\begin{array}{l}\psi_{0 0'\ldots 0'0'}\\ \psi_{1 0'\ldots 0'0'}
\\ \vdots\\ \psi_{01'\ldots 1' 1' }\\ \psi_{1 1'\ldots 1'1' }\end{array}\right)=\left(\begin{array}{l}  \vdots\\ \psi_{0,j}\\ \psi_{1,j}
\\ \vdots\\\psi_{0,k-1 }\\ \psi_{1,k -1}
\end{array}\right)=\left(\begin{array}{l}  \vdots\\ \psi_{2j}\\ \psi_{2j+1}
\\ \vdots\\\psi_{2k-2 }\\ \psi_{2k -1}
\end{array}\right),
\end{equation}
where $\phi_{j }:=\phi_{1' \ldots 1' 0'\ldots   0'}$ with $j$ indices  equal $1' $, $\psi_{A,j }:=\psi_{A,1' \ldots 1' 0'\ldots   0'}$ with $j$ indices
equal $1'$, $A=0,1$, then
the operator  $D_0^{(k)}$ in (\ref{eq:Dirac-Wey-k-0}) can  be written as
a $(2 k)\times (k+1)$-matrix valued differential operator of the first
order from $C^1(\Omega, \mathbb{C}^{k+1} )$ to $ C^0 (\Omega,
\mathbb{C}^{2k } )$ as follows
\begin{equation*} D_0^{(k)}=\left(\begin{array}{rrrrrrr}  - {\partial}_{\overline z_1}&
-{\partial}_{\overline z_0}& 0& 0&0
& \cdots \\ \partial_{ z_0}& -\partial_{z_1}& 0
&0
&0
 &\cdots \\
0& -{\partial}_{\overline z_1}&
- {\partial}_{\overline z_0}&0
&0
&\cdots \\
0&\partial_{ z_0}& -\partial_{z_1}&0
&0
&\cdots
 \\0
&0
& - {\partial}_{\overline z_1}&
-{\partial}_{\overline z_0}
&0
&\cdots \\0
&0
&\partial_{z_0}& -\partial_{ z_1}
&0
 &\cdots \\0
&0&0
&- {\partial}_{\overline z_1}&
-{\partial}_{\overline z_0}
&\cdots \\  \vdots&\vdots&\vdots &\vdots & \vdots&\vdots
 \end{array}\right)
,
\end{equation*}
(cf. \cite{Wa10}) and so
\begin{equation}\label{eq:D0-k-nu-conj}
   D_0^{(k)*}=-
    \left(\begin{array}{rrrrrrrr}-{\partial}_{{z}_1}& {\partial}_{\overline z_0}
 & 0&0&0&0 &0&\cdots\\-{\partial}_{{z}_0}&-{\partial}_{\overline z_1}& - {\partial}_{z_1}&{\partial}_{\overline z_0}
 &0&0 &0&\cdots\\
0&0&-{\partial}_{{z}_0}&-{\partial}_{\overline z_1}& - {\partial}_{z_1}&{\partial}_{\overline z_0}&0 & \cdots \\
0&0& 0&0&-{\partial}_{{z}_0}&-{\partial}_{\overline z_1}& -{\partial}_{ z_1}&\cdots\\\vdots &\vdots &\vdots &  \vdots & \vdots & \vdots &\vdots&\vdots
 \end{array}\right).
\end{equation}
Then
\begin{equation}\label{eq:D0D0-k}
      D_0^{(k) }  D_0^{(k)*}=-\left(\begin{array}{cccccccc} \triangle&
0& \partial_{ \overline{ z}_0} \partial_{ z_1}&-\partial_{ \overline z_0}^2
&0
&0&0&\cdots\\ *&
\triangle& \partial_{z_1}^2
&-\partial_{ \overline z_0} \partial_{z_1}
&0
 &0&0&\cdots\\
 *&*& \triangle&0
&\partial_{ \overline z_0} \partial_{z_1}&-\partial_{\overline  z_0}^2
&0&\cdots\\
* & *& * & \triangle
&\partial_{z_1}^2
&-\partial_{ \overline z_0} \partial_{ z_1}
&0
&\cdots \\
*&*
& *&*
& \triangle
 &0&\partial_{ \overline z_0} \partial_{  z_1}&\cdots\\
*&*
& *&*
 &*
& \triangle &\partial_{ z_1}^2
 &\cdots\\
*& *&*
& *&*
&* & \triangle&\cdots\\ \vdots  &\vdots &\vdots  & \vdots & \vdots & \vdots &\vdots &\vdots
 \end{array}\right)
,
\end{equation}
by direct calculation. Here $* $-entries  are known again by Hermitian symmetry.

By Green's formula (\ref{eq:green00}),
$\psi \in {\rm Dom} D_0^{(k)*}\cap C^1 (\Omega,
   \mathbb{C}^{2k })$ if and only if $ D_0^{(k)*}(\nu)\psi=0$ on the boundary. When $\nu=(1,0,0,0)$   this condition becomes
 \begin{equation*}\begin{split}
  0&=     \left(\begin{array}{ccccccccc} 0 &1
 & 0&0&0 &0&\cdots & 0&0\\-1 &
0& 0&1
 &0 &0&\cdots & 0&0\\
0&0&-1  &
0& 0 &1 &\cdots & 0&0\\
0&0& 0&0&-1  &
0&\cdots & 0&0\\\vdots &\vdots &\vdots & \vdots&\vdots &\vdots&\ddots & \vdots&\vdots\\
 0 &0
 & 0&0&0 &0& \cdots &-1&0
 \end{array}\right) \psi|_{\partial\Omega},
\end{split}\end{equation*}
 from which we get
 \begin{equation}\label{eq:boundary1-k}
     \psi_{1}=\psi_{2k-2}=0,\qquad \psi_{ j}-\psi_{ j+3}=0,\qquad j= 0, 2,4,\ldots, 2k-4.
 \end{equation}

A section $\Psi\in
C^\infty(\Omega,\odot^{k-2 }\mathbb{C}^2 \otimes \Lambda^2\mathbb{C}^2)$ has $(k-1)$  components
$\Psi_{010'\ldots 0'}, \Psi_{011'\ldots 0'}  ,  \ldots , \Psi_{011'\ldots
1'}$. We use notations $\Psi_{ j }:=\Psi_{01 1' \ldots 1' 0'\ldots   0'}$ with $j$ indices equal $1'$ and $\psi_j$ as in (\ref{eq:Dk}).
By definition  $(D_1^{(k) } \psi)_{01B'\ldots C'}=\sum_{A'=0',1'}\left(\nabla_0^{A'}\psi_{ 1A'B'\ldots C'}-\nabla_1^{A'}\psi_{0A'B'\ldots C'}\right)$, and
we have
\begin{equation*}
 \left(D_1^{(k) } \psi\right)_{ j}=-\nabla_1^{0'}\psi_{0,j}+\nabla_0^{0'}\psi_{ 1,j}-\nabla_1^{1'}\psi_{0,j+1}+\nabla_0^{1'}\psi_{ 1,j+1},
\end{equation*} $j= 0, \ldots,  k-2.$
We see that $\Psi=D_1^{(k) }\psi$ with $(k-1)\times (2k)$-matrix operator
\begin{equation}\label{eq:D1-k}\begin{split}
     D_1^{(k) }& = \left(\begin{array}{cccccccc}-\nabla_1^{0'}&\nabla_0^{0'}&-\nabla_1^{1'}&\nabla_0^{1'}&0&0&0&\cdots\\
    0&0& -\nabla_1^{0'}&\nabla_0^{0'}&-\nabla_1^{1'}&\nabla_0^{1'} &0&\cdots\\
    0&0&0&0& -\nabla_1^{0'}&\nabla_0^{0'} &- {\nabla}_1^{1'}&\cdots\\ \vdots &\vdots &\vdots &\vdots &\vdots &\vdots&\vdots &\vdots
\end{array}\right)\\&  = \left(\begin{array}{rrrrrrrr}-{\partial}_{{z}_0}&-{\partial}_{\overline z_1}& {\partial}_{z_1}& - {\partial}_{\overline
z_0}&0&0&0&\cdots\\
    0&0& -{\partial}_{{z}_0}&-{\partial}_{\overline z_1}& {\partial}_{z_1}& - {\partial}_{\overline z_0}&0&\cdots\\
    0&0&0&0&-{\partial}_{{z}_0}&-{\partial}_{\overline z_1}& {\partial}_{ z_1}  &\cdots\\ \vdots &\vdots &\vdots &\vdots &\vdots &\vdots&\vdots &\vdots
\end{array}\right), \end{split}\end{equation}
and $ (2k)\times (k-1)$-matrix operator
\begin{equation}\label{eq:D1-k-conj}
     D_1^{(k)* }=-\left(\begin{array}{rrrr} -{\partial}_{\overline z_0} &0&0&\cdots\\ -{\partial}_{  z_1}&0&0 &\cdots \\{\partial}_{\overline  z_1} &
     -{\partial}_{\overline {z}_0}&0&\cdots\\-{\partial}_{ {z}_0}& -{\partial}_{ {z}_1}&0&\cdots\\0&{\partial}_{ \overline z_1} & -{\partial}_{\overline
     {z}_0}&\cdots \\0&-{\partial}_{ {z}_0}& -{\partial}_{ {z}_1}&\cdots\\ \vdots &\vdots &\vdots &\ddots\end{array}\right).
\end{equation}
Thus
\begin{equation}\label{eq:D1D1-k}
     D_1^{(k)*   } D_1^{(k)}=-\left(\begin{array}{ccccccc} {\partial}_{ z_0}{\partial}_{ \overline z_0} & \partial_{ \overline z_0} {\partial}_{\overline
     {z}_1} &-\partial_{\overline z_0} {\partial}_{ { z}_1} & \partial_{\overline z_0}^2&0 &0&\cdots\\ * & \partial_{\overline z_1} {\partial}_{  { z}_1}&
     -  {\partial}_{ { z}_1}^2 &\partial_{\overline z_0} {\partial}_{  {{z}}_1}& 0&0 &\cdots \\  * & *  &\triangle &0 &  -\partial_{\overline z_0}
     {\partial}_{ { z}_1} & \partial_{\overline z_0}^2&\cdots\\  * &   * &  * & \triangle&-  {\partial}_{ { z}_1}^2 &\partial_{\overline z_0} {\partial}_{
     {{z}}_1} &\cdots\\  * &  * & *  & *  &\triangle &0&\cdots \\  * &  * &  * &  * &  * & \triangle&\cdots\\ \vdots &\vdots &\vdots &\vdots &\vdots
     &\vdots&\vdots  \end{array}\right).
\end{equation}
  The sum of (\ref{eq:D0D0-k}) and (\ref{eq:D1D1-k})
gives
\begin{equation}\label{eq:laplace-k}
\square_1^{(k) }:=  D_0^{(k) }  D_0^{(k)*} + D_1^{(k)*   } D_1^{(k)}=-\left(\begin{array}{ccccccc} \Delta+\Delta_1 &    L &0 &\cdots& 0&0 & 0\\
\overline{ L }& \Delta+ \Delta_2&0 &\cdots& 0& 0 &0\\ 0 &0&2\triangle&\cdots&0 & 0&0 \\ \vdots &\vdots&\vdots &\ddots&\vdots & \vdots&\vdots \\ 0
&0&0&\cdots&2\triangle & 0&0  \\  0 &0 & 0& \cdots&0& \Delta+\Delta_2&   -  L  \\0&0&0 &\cdots& 0&  - \overline{ L }& \Delta+\Delta_1\end{array}\right).
\end{equation}
This is an elliptic operator.
Using the notation in (\ref{eq:zeta-nu}), we obtain the $(2k) \times (k+1)$-matrix
\begin{equation}\label{eq:D0-k-nu}D_0^{(k)}(\nu) =\left(\begin{array}{cccccc} -\overline{\zeta_1}&
-\overline{\zeta_0}& 0&0
&0
&\cdots \\   {\zeta_0}& - {\zeta_1}& 0
&0
&0
 &\cdots \\
0&-\overline{\zeta_1}&
-\overline{\zeta_0}&0
&0
&\cdots \\
0&  {\zeta_0}& - {\zeta_1}&0
&0
&\cdots
 \\0
&0
&-\overline{\zeta_1}&
-\overline{\zeta_0}
&0
&\cdots \\0
&0
&  {\zeta_0}& - {\zeta_1}
&0
 &\cdots \\0
&0&0
&-\overline{\zeta_1}&
-\overline{\zeta_0}
&\cdots \\  \vdots&\vdots&\vdots &\vdots & \vdots&\ddots
 \end{array}\right),
 \end{equation}
the $ (k-1) \times(2k) $-matrix
\begin{equation}\label{eq:D1-k-nu}
     D_1^{(k) } (\nu) = \left(\begin{array}{cccccccc}- {\zeta_0}&- \overline{\zeta_1}& {\zeta_1}&
-\overline {\zeta_0}&0&0&0&\cdots\\
    0&0& -{\zeta_0}&- \overline{\zeta_1}&   {\zeta_1}&
- \overline{\zeta_0} &0&\cdots\\
    0&0&0&0& - {\zeta_0}&- \overline{\zeta_1}&   {\zeta_1}&
 \cdots\\ \vdots &\vdots &\vdots &\vdots &\vdots &\vdots&\vdots &\vdots
\end{array}\right), \end{equation}
and the $(2k)\times (k-1) $-matrix
\begin{equation}\label{eq:D1-k-nu-conj}
     D_1^{(k) *} (\nu)  =- \left(\begin{array}{rrrrr}-\overline {\zeta_0}&0&0&0&\cdots
     \\-{\zeta_1}&0&0&0&\cdots\\
    \overline  {\zeta_1}&-\overline {\zeta_0}&0&0&\cdots\\
    -{\zeta_0}&-{\zeta_1}&0&0&\cdots\\
    0& \overline{\zeta_1}&0&0&\cdots\\ 0&-{\zeta_0}&- \overline{\zeta_0}  &0&\cdots\\
     0&0& { \zeta_1}&0&\cdots \\ \vdots&\vdots&\vdots &\vdots & \ddots
     \end{array}\right). \end{equation}

By Green's formula (\ref{eq:green00}),
$\Psi  \in {\rm Dom} D_1^{(k)*  } \cap C^1 (\Omega,  \mathbb{C}^{k -1}  )$ if and only if $ D_1^{(k)* } (\nu)\Psi|_{\partial\Omega}=0$ on the boundary.
    It follows from $D_1^{(k)* } (\nu)$ in (\ref{eq:D1-k-nu-conj}) that $ \Psi|_{\partial\Omega}=0$  since $\zeta_0$ and $\zeta_1$ can not vanish
    simultaneously.
 Now $D_1^{(k)  } \psi\in {\rm Dom}D_1^{(k)* } \cap C^1 (\Omega,\mathbb{C}^{k -1} )$  if and only if $ D_1^{(k)  } \psi=0$ on the boundary. So
 (\ref{eq:bvp}) is our natural boundary value condition.
\subsection{The  boundary value problem
(\ref{eq:bvp}) satisfies the Shapiro-Lopatinskii condition}

Suppose $u(t )$ is a rapidly decreasing solution
on $[0,\infty)$ to the following ODE  under the initial condition:
 \begin{equation}\label{eq:Lopatinski-Shapiro-K-k}\left\{\begin{split}&
(D_0^{(k)}D_0^{(k)*}+  D_1^{(k)* } D_1^{(k)})(i\xi +\nu \partial_t)u( t)=0,\\& D_0^{(k)*}  (\nu)u(0)=0,\\&
D_1^{(k)*} (\nu)D_1^{(k)}(i\xi +\nu\partial_t)u(0)=0.\end{split}\right.\end{equation}
  Define a
function $U:{\mathscr V}_\nu\rightarrow  \mathbb{ C}^{2k}$
as in (\ref{eq:U}). Now let us show that $U$ vanishes, which will prove that the boundary value problem (\ref{eq:Lopatinski-Shapiro-K-k}) satisfies  the
Lopatinski-Shapiro condition.
  By using arguments as in the case $k=2$ and the following Proposition 4.1, we find that
 \begin{equation}\label{eq:harmonic-BVP-k}\left\{\begin{array}{l}  \Delta U =0,\qquad {\rm on }\quad {\mathscr V}_\nu,\\
D_0^{(k)*  } (\nu) U|_{\partial {\mathscr V}_\nu}=0,\\ D_1^{(k) * } (\nu) D_1^{(k)  } U|_{\partial {\mathscr V}_\nu}=0,
 \end{array} \right.
   \end{equation}

It is direct to check that $D_1^{(k) }(\nu)D_0^{(k) }(\nu)=0$, which can be also obtained  from $D_1^{(k) } D_0^{(k) } =0$. The matrix
$D_0^{(k) }(\nu)$ in (\ref{eq:D0-k-nu}) has rank $ k+1$, and  $
D_1^{(k) }(\nu)$ in (\ref{eq:D1-k-nu})  has rank $k-1$. Moreover,
${\rm Im}D_0^{(k) }(\nu)=\ker D_1^{(k) }(\nu)$ and the space ${\rm Im}D_1^{(k) *}(\nu) $ is $(k-1)$-dimensional,
  orthogonal to $\ker D_1^{(k) }(\nu)$. Namely we have the orthogonal  decomposition
\begin{equation*}
   \mathbb{C}^{2k}={\rm Im}D_0^{(k) } (\nu)\oplus{\rm Im}D_1^{(k)* } (\nu) \cong \mathbb{C}^{ k+1}\oplus \mathbb{C}^{k-1}.
\end{equation*}
We rewrite $U$ as
\begin{equation*}
    U= D_0^{(k) } (\nu)U'+ D_1^{(k)* } (\nu) U'',
\end{equation*} for some $ \mathbb{C}^{ k+1}$-valued function $U'$ and $ \mathbb{C}^{ k-1}$-valued  function $U''$.
 Then,
\begin{equation*}
     D_0^{(k)* }  (\nu) U= D_0^{(k) *}(\nu)   D_0^{(k) } (\nu) U' .
\end{equation*}
Here $ D_0^{(k) *}  (\nu) D_0^{(k) } (\nu)$
is an invertible $(k+1)\times(k+1)$-matrix because $ D_0^{(k) } (\nu)$ has rank $k+1$.  Consequently, $U'$ and $U''$ are both harmonic. The second
equation in (\ref{eq:harmonic-BVP})
implies that $U' =0$ on the boundary $\partial {\mathscr V}_\nu$, and so it vanishes as a harmonic function  on the whole half space ${\mathscr V}_\nu$.
Now we have $ U=D_1^{(k) *} (\nu) U''$.

The third  equation in (\ref{eq:harmonic-BVP-k})
implies that    $D_1^{(k) } U|_{\partial {\mathscr V}_\nu}=0$ by $D_1^{(k)* } (\nu)$ in (\ref{eq:D1-k-nu-conj}).   Note that
\begin{equation*}
     ( -{\partial}_{{z}_0},-{\partial}_{\overline z_1}, {\partial}_{z_1},- {\partial}_{\overline z_0})\left(\begin{array}{r } -\overline{\zeta_0} \\-
     {\zeta_1}  \\ \overline{\zeta_1} \\
- {\zeta_0 }\end{array}\right)=2\partial_\nu
\end{equation*}
as in (\ref{eq:nu-der}),
and
\begin{equation*}\begin{split}
    \mathscr L:= (-{\partial}_{ z_0},- {\partial}_{\overline z_1}) \left(\begin{array}{r }\overline{\zeta_1}  \\- {\zeta_0 }  \end{array}\right)&= -(
    \partial_{x_{0}}- {i} \partial_{x_{ 1}})(\nu_2+ {i}\nu_3 )
+(\partial_{x_{ 2}}+ {i}\partial_{x_{ 3}})(\nu_0-{i}\nu_1)
=\partial_{\mu}+i\partial_{\widetilde{\mu}},
\end{split}\end{equation*}
where
\begin{equation}\label{eq:mu}
\mu=(-\nu_2,-\nu_3,   \nu_0,   \nu_1) ,  \qquad \widetilde{\mu}=(-\nu_3,\nu_2, -\nu_1,   \nu_0 )   ,
\end{equation}
and
\begin{equation*}
     ( {\partial}_{{z}_1},-{\partial}_{\overline z_0})\left(\begin{array}{r } - \overline{\zeta_0} \\
- {\zeta_1 }\end{array}\right)=- \overline{\mathscr L}.
\end{equation*}
Then   we find that
\begin{equation}\label{eq:U-k}
\begin{split}
D_1^{(k) } U& =D_1^{(k) } D_1^{(k) * } (\nu) U''=\left(\begin{array}{ccccccc}2 \partial_\nu& - \overline{\mathscr L}&0 &\cdots& 0&0 & 0\\ \mathscr     { L
}& 2\partial_\nu& - \overline{\mathscr L} &\cdots& 0& 0 &0\\ 0 &\mathscr     { L }& 2\partial_\nu&\cdots&0 & 0&0 \\ \vdots &\vdots&\vdots &\ddots&\vdots &
\vdots&\vdots   \\  0 &0 & 0& \cdots&\mathscr     { L }& 2\partial_\nu & - \overline{\mathscr L}\\0&0&0 &\cdots& 0& \mathscr     { L }&
2\partial_\nu\end{array}\right)
 \left(\begin{array}{c}
  U_1''\\ \vdots\\ U_{k-1}''
 \end{array}\right)=0
\end{split}
\end{equation}
on the boundary $\partial   {\mathscr V}_\nu$, by using $D_1^{(k) } $ in (\ref{eq:D1-k}) and $ D_1^{(k) * } (\nu) $ in (\ref{eq:D1-k-nu-conj}).

When $k=3$, we obtain
 \begin{equation} \label{eq:U-4-5}\left\{ \begin{array}{r }2\partial_\nu {U}_{1}''-(\partial_\mu  - i\partial_{\widetilde{\mu}} )  {U}_{ 2}''=0,\\
(\partial_\mu +i\partial_{\widetilde{\mu}} ) {U}_{ 1}''+2\partial_\nu  {U}_{2 }''=0,
 \end{array} \right.
 \end{equation}on the boundary $\partial {\mathscr V}_\nu$. Both $2\partial_\nu {U}_{1}''-(\partial_\mu  - i\partial_{\widetilde{\mu}} )  {U}_{ 2}''$ and
 $
(\partial_\mu +i\partial_{\widetilde{\mu}} ) {U}_{ 1}''+2\partial_\nu  {U}_{2 }''$ are harmonic functions on $  {\mathscr V}_\nu$, and so must  vanish.
Namely, (\ref{eq:U-4-5}) holds on the whole half space $  {\mathscr V}_\nu$.
On the other hand, as a harmonic function, $\triangle U=e^{ix\cdot\xi} (u''-|\xi|^2u)(x\cdot\nu)=0 $. So as a rapidly decreasing function, we must have
$u(t)=e^{-|\xi|t}u_0$ for some vector $u_0\in\mathbb{ C}^{6}$. Consequently,  $U''=e^{ix\cdot\xi-|\xi|x\cdot\nu} W'' $ for some vector $W''\in\mathbb{
C}^{ 2}$.
Then substitute $U''$ into~\eqref{eq:U-4-5} to get
\begin{equation*} \left(\begin{array}{cc}-2|\xi|&\overline{ \Lambda}\\ \Lambda&-2|\xi| \end{array}\right)
 \left(\begin{array}{c}
  W_1''\\   W_{2}''
 \end{array}\right) =0
 \end{equation*}
where  $\Lambda= i (\mu \cdot\xi + i\widetilde{ \mu} \cdot\xi)$.
\begin{equation*}
     \det \left(\begin{array} {rr}-2|\xi|& \overline{\Lambda}\\ \Lambda& -2|\xi|\end{array}\right)=4|\xi|^2-|\Lambda|^2>0,
\end{equation*} by $ |\Lambda|\leq |\xi|$ since $\mu$ and $\widetilde{  \mu}$ are mutually orthogonal unit vectors in the hyperplane orthogonal to $\nu$
(cf. (\ref{eq:mu})), and $\xi\perp\nu$. Hence $W''=0$ and $U$ vanishes.

 In the case $k>3$,  $U''=e^{ix\cdot\xi-|\xi|x\cdot\nu} W'' $ for some vector $W''\in\mathbb{ C}^{k-1}$.
Substituting $U''$ into~\eqref{eq:U-k}, we get
\begin{equation} \label{eq:U-k_1} \left(\begin{array}{rrrrrrrrr } -2|\xi|&\overline{ \Lambda}&0&0&0&0&0&\cdots\\
{ \Lambda}&-2|\xi|& \overline { \Lambda}&0&0&0&0&\cdots\\0&{ \Lambda}&-2|\xi|& \overline { \Lambda}&0&0&0&\cdots\\
\vdots&\vdots&\vdots&\vdots&\vdots&\vdots&\vdots&\vdots
 \end{array} \right)W''=0.
 \end{equation}
Observe that, as previously, this condition also holds on the whole half space $  {\mathscr V}_\nu$.
 It suffices to show that
the determinant of the above matrix vanishing.  This is true because the determinant equals to
 \begin{equation*} \det\left(\begin{array}{rrrrrrrrr } -2|\xi|&\overline{ \Lambda}&0&0&0&0&0&\cdots\\
0&\lambda'& \overline { \Lambda}&0&0&0&0&\cdots\\0&{ \Lambda}&-2|\xi|& \overline { \Lambda}&0&0&0&\cdots\\
\vdots&\vdots&\vdots&\vdots&\vdots&\vdots&\vdots&\vdots
 \end{array} \right) =\det\left(\begin{array}{rrrrrrrrr } -2|\xi|&\overline{ \Lambda}&0&0&0&0&0&\cdots\\
0&\lambda'& \overline { \Lambda}&0&0&0&0&\cdots\\0&0&\lambda''& \overline { \Lambda}&0&0&0&\cdots\\
\vdots&\vdots&\vdots&\vdots&\vdots&\vdots&\vdots&\vdots
 \end{array} \right)\end{equation*} with $\lambda'=- |\xi|(2-\frac {|\Lambda|^2}{2|\xi|^2})< - |\xi|$ by $ |\Lambda|\leq |\xi|$ again. Then $ \lambda''=- |\xi|(2-\frac
 {|\Lambda|^2}{|\lambda'||\xi| })< - |\xi|  $ if $\lambda' < - |\xi| $. Repeating this procedure, we see that the above determinant   is nonzero.
So $W''=0$ and $U$ vanishes. We complete the proof of the regularity of the boundary value problem (\ref{eq:Lopatinski-Shapiro-K-k}).

\begin{lem} \label{lem:exact-k} The  sequence
\begin{equation*}
   0\leftarrow  \mathbb{C}^{k+1}\xleftarrow{
D_0^{(k) }(\xi)^t}\mathbb{C}^{2k} \xleftarrow{ D_1^{(k) } (\xi)^t}\mathbb{C}^{k-1}\leftarrow 0
\end{equation*}
 is
exact for any nonzero $\xi\in \mathbb{R}^4 $.
\end{lem}\begin{proof}
Let $\eta$ be as in (\ref{eq:eta}),
\begin{equation*}
     D_0^{(k) } (\xi)^t=\frac 1i
    \left(\begin{array}{rrrrrrrr}-{\overline{\eta}_1}& { \eta_0}
 & 0&0&0&0 &0&\cdots\\-{\overline {\eta}_0}&-{\eta_1}& -\overline {\eta_1}&{ \eta_0}
 &0&0 &0&\cdots\\
0&0&-{\overline{\eta}_0}&-{ \eta_1}& -\overline {\eta_1}&{ \eta_0}&0 & \cdots \\
0&0& 0&0&-{\overline{\eta}_0}&-{ \eta_1}& -{\overline \eta_1}&\cdots\\\vdots &\vdots &\vdots &  \vdots & \vdots & \vdots &\vdots&\ddots
 \end{array}\right)
\end{equation*}
and
\begin{equation*}
     D_1^{(k) } (\xi)^t=\frac 1i\left(\begin{array}{rrrrr}-{\eta_0}&0&0&0&\cdots
     \\-\overline {\eta_1}&0&0&0&\cdots\\
{\eta_1}&- {\eta_0}&0&0&\cdots\\
    -    \overline {\eta_0}&-\overline{\eta_1}&0&0&\cdots\\
    0&{\eta_1}&0&0&\cdots\\ 0&-\overline{\eta_0}&-{\eta_0}  &0&\cdots\\
     0&0&\overline { \eta_1}&0&\cdots\\   \vdots & \vdots & \vdots &\vdots&\ddots
     \end{array}\right).
\end{equation*}
The proof of the equality ${\rm Im}D_1^{(k) }(\xi)^t=\ker D_0^{(k) }(\xi)^t$ follows as in the case of $k=2$.
\end{proof}

\begin{prop}\label{prop:exact-k}  The sequence $ \mathscr{R}^{k+1}\xleftarrow{
 D_0^{(k) } (\xi)^t}\mathscr{R}^{2k} \xleftarrow{ D_1^{(k) } (\xi)^t}\mathscr{R}^{k-1}$ is
exact.
\end{prop}
\begin{proof}   Suppose $
 D_0^{(k) } (\xi)^t\left(\begin{array}{c } p_1(\xi)\\\vdots\\ p_{2k}(\xi) \end{array}\right)=0$, where $p_j$ are polynomials. For each $\xi\neq 0$, there
 exists  a unique $f_\xi= (f_{\xi;1},\ldots, f_{\xi;k-1})^t\in\mathbb{C}^{k-1}$, such  that
\begin{equation*}
     \left(\begin{array}{c } p_1(\xi)\\\vdots\\ p_{2k}(\xi) \end{array}\right)= D_1^{(k) } (\xi)^t f_\xi=\frac 1i \left(\begin{array}{ccccc}-  \xi_0+
     i\xi_1&0&0&0&\cdots
     \\- \xi_2-i\xi_3&0&0&0&\cdots\\
 \xi_2-i\xi_3&-  \xi_0+ i\xi_1&0&0&\cdots\\
   - \xi_0-i \xi_1 &- \xi_2-i\xi_3&0&0&\cdots\\
    0& \xi_2-i\xi_3&0&0&\cdots\\ 0&    - \xi_0-i \xi_1&-  \xi_0+ i\xi_1 &0&\cdots\\
     0&0&- \xi_2-i\xi_3&0&\cdots\\   \vdots & \vdots & \vdots &\vdots&\ddots
     \end{array}\right)
       \left(\begin{array}{c } f_{\xi;1}\\ \vdots\\ f_{\xi;k-1} \end{array}\right).
\end{equation*}
In the same way as in the case $k=2$, we can show that
 $f_{\xi;1}$ is a polynomial. Then repeat  this procedure for $ f_{\xi;2} ,    f_{\xi;3}, \ldots$.
\end{proof}
\section{Proofs of main theorems}

\subsection{More about the operator $\square_2^{(k)}$ }
It is direct to see that
\begin{equation*}
     \square_2^{(k)}= D_1^{(k)
      }D_1^{(k)* }= \left(\begin{array}{rrr} 2\triangle&0 &\cdots
     \\0&2 \triangle &\cdots\\ \vdots&\vdots&\vdots
     \end{array}\right),
\end{equation*}
by (\ref{eq:D1-k})-(\ref{eq:D1-k-conj}).  The condition $ D_1^{(k)* } (\nu)\Psi=0$ on the boundary $\partial\Omega$ implies that $
\Psi|_{\partial\Omega}=0$ as before. The boundary value problem
\begin{equation}\label{eq:bvp-k}
  \left\{\begin{array}{l}  \square_2^{(k)}\Psi=0,\qquad {\rm on }\quad \Omega,\\
 \Psi|_{\partial \Omega}=0,
 \end{array} \right.
\end{equation}
 is just the boundary value problem for the usual Laplacian operator with Dirichlet boundary value. It is always solvable by the solution operator $
 N^{(k)}_2: H^s(\Omega,  \mathbb{C}^{ k -1} )\rightarrow H^{s+2}(\Omega,  \mathbb{C}^{ k-1} )$. Consequently,
we have
\begin{equation*}
     \Psi= D_1^{(k)
      }D_1^{(k)* } N^{(k)}_2\Psi
\end{equation*}
and the equation
\begin{equation*}
     D_1^{(k)
      }\psi=\Psi,
\end{equation*}
is uniquely solved by $\psi=D_1^{(k)* } N^{(k)}_2\Psi$  for any $\Psi\in H^s(\Omega,  \mathbb{C}^{ k -1} )$.

\subsection{The  Fredholm property}

\begin{thm} \label{thm:Fredholm}$($\cite[Proposition 11.14 and 11.16]{Ta},~\cite[Theorem 20.1.8 ]{Hor}$)$ Suppose that the boundary value problem
(\ref{eq:bvp-general}) is regular. Then
 the operator
  \begin{equation*}
  T:H^{m+s}(\Omega, E_0)\longrightarrow H^{m }(\Omega, E_1)\oplus\bigoplus_{j=1}^lH^{m +s-m_j-\frac 12}(\partial\Omega, G_j),
\end{equation*} $s=0,1,\ldots$,
 defined by
\begin{equation*}
     T u=(P(x,\partial)u, B_1(x,\partial)u,\ldots,B_l(x,\partial)u)
\end{equation*}
is Fredholm, and satisfies the estimate
\begin{equation}\label{eq:estimateFr}
     \|u\|_{  H^{m+s}(\Omega)}^2\leq C\left( \|Pu\|_{  H^{s }(\Omega)}^2+ \sum_{j=1}^l  \|B_ju\|_{  H^{m +s-m_j-\frac 12}(\partial\Omega)}^2+\|u\|_{
     H^{m-1 }(\Omega)}^2\right)
\end{equation}
for some positive constant $C $. Moreover, the kernel and the space orthogonal to the range consist of smooth functions.
     \end{thm}

  By adding the boundary value condition (\ref{eq:bvp}), we consider the   closed subspace $ H^s_b(\Omega,\mathbb{C}^{2k})$  of Sobolev spaces $ H^s
  (\Omega,C^{2k})$ defined by
     \begin{equation*}
         H^s_b(\Omega,\mathbb{C}^{2k}):=\left\{u\in H^s (\Omega,\mathbb{C}^{2k});D_0^{(k) *}(\nu) u=0, D_1^{(k) *}(\nu)D_1^{(k)  } u=0 \hskip 3mm {\rm on}
         \hskip 3mm \partial\Omega\right\},
     \end{equation*}
     $s>\frac 32$. The boundary value conditions above are well defined for $s>\frac 32$ by the Trace Theorem.

We know that the associated Laplacian $\square_1^{(k) }$ in (\ref{eq:laplace-k})  is an elliptic operator.
In sections 3 and 4, we already showed that boundary value problem
(\ref{eq:bvp}) is regular. So we can apply  Theorem \ref{thm:Fredholm}
      to obtain
the Fredholm  operator
  \begin{equation}\label{eq:T}
  T:H^{2+s}(\Omega,  \mathbb{C}^{2k})\longrightarrow H^{s}\left(\Omega, \mathbb{C}^{2k}\right)\oplus H^{ s+ \frac 32}\left(\partial\Omega, \mathbb{C}^{
  k+1} \right  )\oplus H^{ s+ \frac 12}\left(\partial\Omega, \mathbb{C}^{ k-1} \right)
\end{equation}
 defined by
\begin{equation}\label{eq:T'}
     T u=\left( \square_1^{(k)}u,\left. D_0^{(k) *}(\nu) u\right|_{\partial \Omega}  ,\left.D_1^{(k) *}(\nu)D_1^{(k) } u\right|_{\partial \Omega} \right).
\end{equation}
Restricted to  the
closed subspace $  H^{2+s}_b(\Omega,  \mathbb{C}^{2k})\subset H^{2+s} (\Omega,  \mathbb{C}^{2k})$, the operator $T$ gets the form $T u=
\left(\square_1^{(k)}u, 0,0\right)$ for $u\in H^{2+s}_b(\Omega,  \mathbb{C}^{2k})$. Let us prove that the restriction of $T$ is also Fredholm.
 \begin{cor} \label{prop:Fredholm} The operator
  \begin{equation}\label{eq:Fredholm-box}
  \square_1^{(k) }:H^{2+s}_b(\Omega,  \mathbb{C}^{2k})\longrightarrow H^{s}(\Omega, \mathbb{C}^{2k})
\end{equation}
 is   Fredholm.
\end{cor}
\begin{proof} Suppose that $\square_1^{(k) }$ in (\ref{eq:Fredholm-box}) is not  Fredholm. Identifying $H^{s}(\Omega, \mathbb{C}^{2k})$ with the subspace
$\{(f,0,0); f\in H^{s}(\Omega, \mathbb{C}^{2k})\}$ of
\begin{equation*}
\mathscr{W}_s=H^{s}\left(\Omega, \mathbb{C}^{2k}\right)\oplus H^{ s+ \frac 32}\left(\partial\Omega, \mathbb{C}^{ k+1}\right )\oplus H^{  s+ \frac
12}\left(\partial\Omega, \mathbb{C}^{ k-1}\right ),\end{equation*}
we see that the kernel of $ \square_1^{(k) }$ is contained in the kernel of
 the operator $T$ in (\ref{eq:T})-(\ref{eq:T'}), and so its  dimension   must be finite. Thus the cokernel of $\square_1^{(k) }$ should be infinite
 dimensional.

 Let us denote by $M_0$ the subspace of the Hilbert space $H^{s}(\Omega, \mathbb{C}^{2k})
 $ orthogonal to the range of  $\square_1^{(k) }$, and denote by $M $  the subspace of  the Hilbert space
 $\mathscr{W}_s $
orthogonal to the range of $T$.
 Note that $H^{s}(\Omega, \mathbb{C}^{2k})$ is a closed subspace of the Hilbert space $\mathscr{W}_s $ by the above identification, and  the range of $T$
 in  $\mathscr{W}_s $ is closed because it is Fredholm. So as the intersection of $H^{s}(\Omega, \mathbb{C}^{2k})$ and the range of $T$,  the range of
 $\square_1^{(k) }$ is also closed. The space $M $ is of finite dimension. Let $\{v_1,\ldots v_m\}$ be a basis of $M$. Vectors $v_1,\ldots v_m$ define
 linear functionals on $\mathscr{W}_s $, in particular on $M_0$,  by the inner product of  $\mathscr{W}_s $. Because $M_0$ is   infinite dimensional,
 there must be some nonzero  vector $v\in M_0$ in the kernel of these functionals, i.e.,  orthogonal to $M $.   Consequently, $(v,0,0)$ belongs to  the
 range of $T$. Namely, there exists $u\in H^{2+s} (\Omega,  \mathbb{C}^{2k})$ such that $Tu=(v,0,0)$. This also implies that $u\in H^{2+s}_b(\Omega,
 \mathbb{C}^{2k})$ and $\square_1^{(k) }u=v$, i.e., $v$ is in the range of  $\square_1^{(k) }$. This contradicts to $v\in M_0$. Thus
 $\square_1^{(k) }$ has finite dimensional cokernel.   The result follows.
\end{proof}

 \subsection{Proofs of main theorems}

  {\it Proof of Theorem     \ref{thm:BVP}}. It is sufficient to prove the theorem for $s=0$. By Corollary \ref{prop:Fredholm},
 the map
$
  \square_1^{(k) }:H^{2 }_b(\Omega,  \mathbb{C}^{2k})\longrightarrow L^2(\Omega, \mathbb{C}^{2k})
$
 is   Fredholm.
  So its kernel, denoted by $\mathscr K$, is finite dimensional.  Denote by $\mathscr K^{ \perp }  $   the
  orthogonal complement to $ \mathscr{K}$ in $H^{2 }_b(\Omega,  \mathbb{C}^{2k})$ under the inner product of $H^{2 }_b(\Omega,  \mathbb{C}^{2k})$. Denote
  by $\mathscr{R}$ the range of $\square_1^{(k) }$ in $ L^2 (\Omega,  \mathbb{C}^{2k})  $. It is a closed subspace since the cokernel of $\square_1^{(k)
  }$  is also finite dimensional.  Then $  \square_1^{(k) }:\mathscr K^{ \perp  } \rightarrow\mathscr{R}$ is bijective, and so there exists a inverse
  linear operator $\widetilde{N}_1^{(k)   }:\mathscr{R}\rightarrow\mathscr K^{ \perp  }$. As the Fredholm operator, $
  \square_1^{(k) }:H^{2 }_b(\Omega,  \mathbb{C}^{2k})\longrightarrow \mathscr{R}
$
 is bounded, so is its inverse $\widetilde{N}_1^{(k)   }$ by the inverse operator theorem.
Moreover, $\widetilde{N}_1^{(k) }$ can be extended to a bounded operator
\begin{equation}\label{eq:N-1-(k)}
     {N}_1^{(k) }: L^2(\Omega, \mathbb{C}^{2k} )\longrightarrow\mathscr K^{ \perp  }\subset H^{2 }_b(\Omega,  \mathbb{C}^{2k})
\end{equation}
  by setting $N_1^{(k) }$ vanishing on $ \mathscr{R}^{ \perp  }$, the space orthogonal to $ \mathscr{R} $ in $ L^2(\Omega,  \mathbb{C}^{2k})  $ under the
  $L^2$ inner product. Namely,
\begin{equation*} {N}_1^{(k)   }f=\left\{\begin{array}{l}
\widetilde{N}_1^{(k)   }f,\qquad {\rm if}\quad f\in \mathscr R,\\ 0,\qquad \qquad{\rm if}\quad f\in \mathscr R^{ \perp  }.
 \end{array}\right.
\end{equation*}
Moreover, there exists a positive constant $C$ such that
\begin{equation}\label{eq:eestimate-H2-L2}
     \|N_1^{(k) }f\|_{H^{ 2}(\Omega, \mathbb{C}^{2k})}\leq C \|f\|_{L^2(\Omega, \mathbb{C}^{2k})}
\end{equation}
 for any $f\in L^2(\Omega, \mathbb{C}^{2k})$.

Now we can establish the Hodge-type orthogonal decomposition following the ideas~\cite[chapter 5 \S 9]{Ta} for De Rham complex. By using   the identity
(\ref{eq:green0}) in Corollary \ref{cor:green0} twice, we see that if  $\varphi,\varphi'\in H^{ 2}_b(\Omega, \mathbb{C}^{2k})$, then
\begin{equation}\label{eq:square-int}\begin{split}
  \left  (\square_1^{(k) }\varphi,\varphi'\right)&=\left(\left(D_0^{(k) }D_0^{(k)* } + D_1^{(k) *}D_1^{(k) }\right)\varphi,\varphi'\right)
   \\& =\left(D_0^{(k)* }  \varphi,D_0^{(k)* }\varphi'\right)+\left(D_1^{(k) } \varphi,D_1^{(k)  }\varphi'\right)\\&
   =\left(\varphi,\left(D_0^{(k) }D_0^{(k)* } + D_1^{(k) *}D_1^{(k) }\right)\varphi'\right)=\left (\varphi,\square_1^{(k) }\varphi'\right),
\end{split}\end{equation}
 since $D_0^{(k)* }(\nu)\varphi'|_{\partial\Omega}=D_1^{(k)* }(\nu)D_1^{(k) }\varphi|_{\partial\Omega} =0$ and $D_0^{(k)* }(\nu)\varphi|_{\partial\Omega}
 =D_1^{(k)* }(\nu)D_1^{(k) }\varphi'|_{\partial\Omega} =0$.

We show that $\widetilde{N}_1^{(k) }$ is a self adjoint operator on $\mathscr{R}$. For any $u ,v \in \mathscr{R}$, we can write $u=\square_1^{(k)
}\varphi,v=\square_1^{(k) }\varphi'\in \mathscr{R}$
 for some $\varphi,\varphi'\in H^{ 2}_b(\Omega, \mathbb{C}^{2k})$. Then by using (\ref{eq:square-int}),
 \begin{equation*}\begin{split}
    \left (\widetilde{N}_1^{(k) } u,v\right)&=\left (\widetilde{N}_1^{(k) }\square_1^{(k) }\varphi, \square_1^{(k) }\varphi'\right)=\left(  \varphi,
    \square_1^{(k) }\varphi'\right)=\left(\square_1^{(k) }  \varphi,\varphi'\right)
     =\left( u,\widetilde{N}_1^{(k) }v\right).
 \end{split}\end{equation*}
  Consequently, $N_1^{(k) }$, as a trivial extension of $\widetilde{N}_1^{(k) }$, is also a self adjoint operator on $L^2(\Omega, \mathbb{C}^{2k})$.
Because of the estimate (\ref{eq:eestimate-H2-L2}), $N_1^{(k) }$ is compact on $L^2(\Omega, \mathbb{C}^{2k})$  by Rellich's theorem.
Hence there is an orthonormal basis $\{u_j\}_{j=1}^\infty$ of $\mathscr{R}\subset L^2(\Omega, \mathbb{C}^{2k})$  consisting of eigenfunctions
of  $N_1^{(k) }$:
\begin{equation*}
      N_1^{(k) } u_j=\lambda_j u_j,\qquad \lambda_j\searrow 0.
\end{equation*} Here $\lambda_j\neq 0$ since $N_1^{(k) }$ is the inverse of $\square_1^{(k) }:\mathscr K^{ \perp  } \rightarrow\mathscr{R}$.
 In the view of   (\ref{eq:N-1-(k)}),
 \begin{equation}\label{eq:u-H2}
    u_j\in H^{2 }_b(\Omega,  \mathbb{C}^{2k})  \qquad {\rm for} \quad {\rm each }\quad j.\end{equation}
Obviously,
 \begin{equation*}
      \square_1^{(k) }u_j=\frac 1{\lambda_j} u_j.
 \end{equation*} Then any element of $\mathscr K^{ \perp  }$ can be written as $\sum_{j=1}^\infty\lambda_j a_ju_j$ for some $a_j$'s with
 $\sum_{j=1}^\infty|a_j|^2<\infty$.
Denote by $u^0_l\in H^{2}_b(\Omega, \mathbb{C}^{2k}) $, $l=1,\ldots, \dim \mathscr K$, a  basis of $\mathscr K$. Then $\{u_j\} \cup \{u_l^0\}$ is a basis
of $ H^{2}_b(\Omega, \mathbb{C}^{2k}) $. Because $C^\infty_0( {\Omega},\mathbb{C}^{2k})\subset H^{2 }_b(\Omega,  \mathbb{C}^{2k}) $ and $C^\infty_0(
{\Omega},\mathbb{C}^{2k})$ is dense in  $L^2(\Omega, \mathbb{C}^{2k})$, we see that $H^{2 }_b(\Omega,  \mathbb{C}^{2k}) $ is dense in  $L^2(\Omega,
\mathbb{C}^{2k})$. So $\{u_j\} \cup \{u_l^0\}$ is also a  basis of $L^2(\Omega, \mathbb{C}^{2k})$. Consequently,
\begin{equation}\label{eq:decomposition0}
   L^2(\Omega, \mathbb{C}^{2k})=\mathscr K\oplus\mathscr{R}.
\end{equation}

 If $\psi\in \mathscr K$, then
 \begin{equation*}\begin{split}
     0 &=  \left( \left(D_0^{(k) }D_0^{(k)* } + D_1^{(k) *}D_1^{(k) }\right)\psi,  \psi \right)= \left( D_0^{(k)* }  \psi,  D_0^{(k)* } \psi \right)
     + \left (   D_1^{(k) } \psi,   D_1^{(k) } \psi \right)
  \end{split}\end{equation*}
  by using  the identity (\ref{eq:green0}) in Corollary \ref{cor:green0}  since  $  \psi \in   H^{2}_b(\Omega, \mathbb{C}^{2k}) $.  Thus  $D_0^{(k)* }
  \psi =0,  D_1^{(k) } \psi=0$. Note that since a function in $  \mathscr K$ is a $C^\infty$ function on $\Omega$ by applying the elliptic estimate
  (\ref{eq:estimateFr}),  we conclude that
  \begin{equation}\label{eq:K-H}
     \mathscr K=\mathscr H^1_{ (k) }(\Omega).
  \end{equation}

 By the construction of the solution operator $N_1^{(k) }$ above and the decomposition (\ref{eq:decomposition0}),
any  $\psi \in H^s(\Omega, \mathbb{C}^{2k}) $  has the Hodge-type  decomposition:
\begin{equation}\label{eq:decomposition}
     \psi=\square_1^{(k) }N_1^{(k) }\psi+P \psi =D_0^{(k) }D_0^{(k)* } N_1^{(k) }\psi+D_1^{(k) *}D_1^{(k) }N_1^{(k) }\psi +P  \psi,
\end{equation} where $P $ is the orthonomal projection to $ \mathscr K=\mathscr H^1_{ (k) }(\Omega)$ with respect to the $L^2$ inner product.

It is sufficient to prove orthogonality of first two terms in (\ref{eq:decomposition})  for smooth functions, since $C^\infty(\overline{\Omega},
\mathbb{C}^{2k}) $ is dense in $L^2(\Omega, \mathbb{C}^{2k}) $ and operators $D_0^{(k) }D_0^{(k)* } N_1^{(k) }$ and $D_1^{(k) *}D_1^{(k) }N_1^{(k) }$ are
both bounded in $L^2( {\Omega}, \mathbb{C}^{2k})$.
The   orthogonality follows from
\begin{equation*}
   \left ( D_0^{(k) }D_0^{(k)* } N_1^{(k) }\psi,D_1^{(k) *}D_1^{(k) }N_1^{(k) }\psi\right )=  \left ( D_1^{(k)  } D_0^{(k) }D_0^{(k)* } N_1^{(k)
   }\psi,D_1^{(k) }N_1^{(k) }\psi \right)=0
\end{equation*}
by using  the identity (\ref{eq:green0}) in Corollary \ref{cor:green0} ($D_1^{(k) *}(\nu)D_1^{(k) }N_1^{(k) }\psi|_{\partial\Omega}=0$) for  $u=D_0^{(k)
}D_0^{(k)* } N_1^{(k) }\psi$ $\in H^{1}(\Omega,  \mathbb{C}^{2k})$ and $v= D_1^{(k) }N_1^{(k) }\psi\in H^{2}(\Omega,  \mathbb{C}^{2k})  $ when $\psi\in
H^{1} (\Omega,  \mathbb{C}^{2k}) $,  and using $D_1^{(k)  } D_0^{(k) }=0$.
  The theorem is proved.$\qquad\qquad \qquad\qquad\qquad \qquad\qquad\qquad \qquad \qquad\qquad \qquad\qquad  \qquad\qquad \square$
\vskip 3mm

 {\it Proof of Theorem     \ref{thm:k-CF}}. We claim that if $D_1^{(k) }\psi=0$ and $\psi$ is orthogonal to  $\mathscr H^1_{ (k) }(\Omega)$, then
\begin{equation}\label{eq:D1-eq}
     \phi=D_0^{(k)* }N_1^{(k) }\psi
\end{equation}
satisfies $D_0^{(k) }\phi=\psi$.
Under the condition $D_1^{(k) }\psi=0$,   the second term in the decomposition (\ref{eq:Hodge-decomposition}) vanishes. This is because
\begin{equation*}\begin{split}\left \|D_1^{(k) *}D_1^{(k) }N_1^{(k) }\psi\right\|_{L^2}^2&=  \left (  D_1^{(k) *}D_1^{(k) }N_1^{(k) }\psi,  D_1^{(k)
*}D_1^{(k) }N_1^{(k) }\psi\right)\\&=
       \left (\psi,  D_1^{(k) *}D_1^{(k) }N_1^{(k) }\psi\right)=    \left ( D_1^{(k)  }\psi,D_1^{(k) }N_1^{(k) }\psi\right)=0
 \end{split}\end{equation*}
by using identity (\ref{eq:green0}). Here $ \psi,$ $ D_1^{(k) }N_1^{(k) }\psi\in   H^{s}(\Omega, \mathbb{C}^{2k}) $ ($s\geq 1$), and $N_1^{(k) }\psi\in
H^{2+s }_b(\Omega,  \mathbb{C}^{2k})$ implies that $D_1^{(k) *}(\nu)D_1^{(k) }N_1^{(k) }\psi|_{\partial\Omega}=0$. The second identity comes from the
orthogonality in the Hodge-type  decomposition
(\ref{eq:decomposition}). The claim follows by $P \psi=0$.

The estimate (\ref{eq:estimate1.1})  follows from the estimate for the solution operator $N_1^{(k) }$ in Theorem \ref{thm:BVP}.

Conversely, if $\psi= D_0^{(k)  }\phi$ for some $\phi\in H^{s+1}(\Omega,  \mathbb{C}^{ k +1} ) $. Then
$\psi\perp \mathscr H^1_{ (k) }(\Omega)$. This is because for any $u\in\mathscr H^1_{ (k) }(\Omega)$,
\begin{equation*}
    ( \psi, u)=  \left ( D_0^{(k)  }\phi, u\right)=  \left ( \phi, D_0^{(k)* }u\right)=0
\end{equation*}
by using  the identity (\ref{eq:green0}) in Corollary \ref{cor:green0}  since $ D_0^{(k)* } (\nu)u=0 $ on the boundary and $u$ and $\phi$ are both from
$H^{ 1}(\Omega,  \mathbb{C}^{ k +1} )$.$\qquad\qquad \qquad\qquad\qquad \qquad\qquad\qquad \qquad \qquad\qquad\qquad \qquad\qquad\quad \square$

\centerline{Acknowledgment} This research project was initiated when the authors visited
the National Center for Theoretical Sciences, Hsinchu, Taiwan during January 2013 and the final version of the paper was completed
while the first and the third author visited NCTS during July  2014. They would like to express their profound gratitude
to the Director of NCTS, Professor Winnie Li for her invitation and
for the warm hospitality extended to them during their stay in Taiwan. The third author would like to express his profound gratitude
to   Department of Mathematica in Bergen university
for the warm hospitality  during his visit in the  spring 2014.

 \end{document}